\documentclass[11pt,oneside,a4paper]{article}

\usepackage[english]{babel}
\usepackage[utf8]{inputenc} 
\usepackage{amsmath,amsfonts,amsthm} % Math packages
\usepackage{geometry}		
\usepackage{authblk}
\usepackage{graphicx}
\usepackage{fourier} 
\usepackage{times}
\usepackage{comment}
\usepackage{pgfplots}
\usepackage[T1]{fontenc}
\usepackage{mathrsfs}
\usepackage{latexsym}
\usepackage{amssymb}
\usepackage{float}
\usepackage{extarrows}
\usepackage{caption}
\usepackage{subcaption}
\usepackage{wrapfig}
\usepackage{lipsum}
\usepackage{algpseudocode}
\usepackage{algorithm}
\usepackage{subcaption}
\usepackage{multicol}
\usepackage{array}

\newcommand{\RR}{\mathbb{R}}

\newcommand{\Sb}{\mathbb{S}}

\newcommand{\rd}{\mathrm{d}}

\newcommand{\fd}{\delta^+}

\newcommand{\fa}{\mu^+}
\newcommand{\ba}{\mu^-}
\newcommand{\bd}{\delta^-}
\newcommand{\cd}[1][1]{\delta^{\langle #1 \rangle}}
\newcommand{\Dc}[1][1]{D^{\langle #1 \rangle}}

\newcommand{\di}{d}

\newtheorem{prop}{Proposition}
\newtheorem{remark}{Remark}
\newtheorem{theorem}{Theorem}
\newtheorem{definition}{Definition} 

\numberwithin{equation}{section}
\DeclareMathAlphabet{\mathcal}{OMS}{cmsy}{m}{n}

\usepackage[activate={true,nocompatibility},
            final,tracking=true,
            kerning=true,
            spacing=true,
            factor=1100,
            stretch=10,
            shrink=10]{microtype}
\microtypecontext{spacing=nonfrench}

\title{Linearly implicit structure-preserving schemes for Hamiltonian systems}
\author{Sølve Eidnes$^{1}$, Lu Li$^{1}$, Shun Sato$^{2}$}

\date{%
\today\\
\vspace{22pt}
    $^1$Department of Mathematical Sciences, NTNU, 7491 Trondheim, Norway.\\E-mail: \texttt{solve.eidnes@ntnu.no}, \texttt{lu.li@ntnu.no} (corresponding author).\\%
    $^2$Graduate School of Information Science and Technology, The University of Tokyo, Bunkyo-ku, Tokyo, Japan.\\E-mail: \texttt{shun{\_}sato@mist.i.u-tokyo.ac.jp}.
}

\begin{document}
\maketitle

\begin{abstract}
Kahan's method and a two-step generalisation of the discrete gradient method are both linearly implicit methods that can preserve a modified energy for Hamiltonian systems with a cubic Hamiltonian. These methods are here investigated and compared. The schemes are applied to the Korteweg--de Vries equation and the Camassa--Holm equation, and the numerical results are presented and analysed.

\end{abstract}

\section{Introduction}

The field of geometric numerical integration has garnered increased attention over the last three decades. It considers the design and analysis of numerical methods that can capture geometric properties of the flow of the differential equation to be modelled. These geometric properties are mainly invariants over time; they are conserved quantities such as Hamiltonian energy, angular momentum, volume or symplecticity. Among them the conservation of energy is particularly important for proving the existence and uniqueness of solutions for partial differential equations (PDEs) \cite{taylor115partial}.
 Numerical schemes inheriting such properties from the continuous dynamical system have been shown in many cases to be advantageous, especially when integration over long time intervals is considered \cite{hairer2006geometric}. 

For general non-linear differential equations, one may use a standard fully implicit scheme to solve a problem numerically. Then a non-linear system of equations must be solved at each time step. Typically this is done by the use of an iterative solver where a linear system is to be solved at each iteration. This quickly becomes a computationally expensive procedure, especially since the number of iterations needed in general increases with the size of the system; see a numerical example comparing the computational cost for implicit and linearly implicit methods in \cite{dahlby2011general}. A fully explicit method on the other hand, may over-simplify the problem and lead to the loss of important information, and will often have inferior stability properties. The golden middle way may be found in linearly implicit schemes, i.e. schemes where the non-linear terms are discretized such that the solution at the next time step is found from solving one linear system.

Our aim is to present and analyse linearly implicit schemes with preservation properties. We consider ordinary differential equations (ODEs) that can be written in the form
\begin{equation}\label{original eq}
\begin{split}
\dot{x}&=f(x)=S\nabla H(x),\quad x\in \mathbb{R}^\di,\\
x(0)&=x_{0},
\end{split}
\end{equation} 
where $S$ is a constant skew-symmetric matrix and $H$ is a cubic Hamiltonian function. The famous geometric characteristic for equations like \eqref{original eq} is that the exact flow is energy-preserving,
\begin{equation*}
\frac{d}{dt}H(x)=\nabla H(x)^T\frac{dx}{dt}=\nabla H(x)^TS\nabla H(x)=0,
\end{equation*}
and symplectic if $S$ is the canonical skew-symmetric matrix:
\begin{equation}
\label{symplecticity}
\Psi_{y_0}(t)^T\, S \, \Psi_{y_0}(t) = S,
\end{equation}
where $\Psi_{y_0}(t):=\frac{\partial \varphi_t(y_0)}{\partial y_0}$, with 
$\varphi_t:\mathbb{R}^{\di}\rightarrow \mathbb{R}^{\di}$, $\varphi_t(y_0)=y(t)$
 the flow map of \eqref{original eq} \cite{hairer2006geometric}.
 A numerical one-step method is said to be energy-preserving if $H$ is constant along the numerical solution, and symplectic if the numerical flow map is symplectic. Both the energy-preserving methods and the symplectic methods, the latter of which has the ability to preserve a perturbation of the Hamiltonian $H$ of \eqref{original eq}, have their own advantages. In particular, the energy-preserving property has been found to be crucial in the proof of stability for several such numerical methods, see e.g \cite{fei1995numerical}. However, there is no numerical integration method that can be simultaneously symplectic
and energy-preserving for general Hamiltonian systems \cite{ge88lph}.  In this paper we will focus on energy-preserving numerical integration.

%We will study schemes based on existing methods with geometric properties.
We wish to study and compare two types of existing methods with geometric properties.
The first one is Kahan's method for quadratic ODE vector fields \cite{kahan1993unconventional}, which by construction is linearly implicit, and for which the geometric properties have been studied in \cite{celledoni2012geometric}. Kahan's method has not been extensively studied for solving PDEs so far, with the notable exception \cite{kahan1997unconventional}. This is a one-step method, but we will also give its formulation as a two-step method in this paper, for easier comparison to the other method to be studied. That method, which we call the polarised discrete gradient (PDG) method, is based on the multiple points discrete variational derivative method for PDEs presented by Furihata, Matsuo and coauthors in the papers \cite{matsuo2001dissipative, matsuo2002spatially,matsuo2007new} and the monograph \cite{furihata2011discrete}. A more general framework for such schemes is given by Dahlby and Owren in \cite{dahlby2011general}. With the aim of easing the comparison to Kahan's method, we present here the two-step method of \cite{furihata2011discrete, dahlby2011general} as it looks for ODEs of the form \eqref{original eq}. When Hamiltonian PDEs are considered, by semi-discretizing in space to obtain a system of Hamiltonian ODEs and then applying the PDG method, one may obtain the schemes of the aforementioned references; a specific scheme will depend on the choice of spatial discretization as well as the choices of some functions to be explained in the next section: the polarised energy and the polarised discrete gradient.

This paper is divided into two main parts. In the next chapter, we present the methods in consideration, and give some theoretical results on their geometric properties. In Chapter 3, we present numerical results for the Camassa--Holm equation and the Korteweg--de Vries equation, including analysis of stability and dispersion, comparing the methods.

\section{Linearly implicit schemes}\label{Method introduction}

We will present an ODE formulation of the linearly implicit schemes presented by Furihata, Matsuo and coauthors in \cite{matsuo2001dissipative, matsuo2002spatially, matsuo2007new, furihata2011discrete} and by Dahlby and Owren in \cite{dahlby2011general}. Inspired by the nomenclature of the latter reference, we call these schemes polarised discrete gradient methods. Then we present a special case of this polarisation method in the same framework as Kahan's method, with the goal of obtaining more clarity in comparison of the methods.

\subsection{Polarised discrete gradient methods}

The idea behind the PDG methods is to generalise the discrete gradient method in such a way that a relaxed variant of the preservation property is intact, while nonlinear terms are discretized over consecutive time steps to ensure linearity in the scheme.
Let us first recall the concept of discrete gradient methods. A discrete gradient is a continuous map $\overline{\nabla}H : \mathbb{R}^\di\times \mathbb{R}^\di\rightarrow \mathbb{R}^\di$ such that for any $x, y \in \mathbb{R}^\di$
\begin{equation*}%\label{DG_1}
%\begin{split}
H(y)-H(x)= (y-x)^T\overline{\nabla}H(x,y).
%\end{split}
\end{equation*}
The discrete gradient method for \eqref{original eq} is then given by
\begin{equation*}%\label{DG_1_method}
\frac{x^{n+1}-x^n}{\Delta t}=S\overline{\nabla}H(x^n,x^{n+1}),
\end{equation*} 
which will preserve the energy of the system \eqref{original eq} at any time step. Here and in what follows, $x^n$ is the numerical approximation of $x$ at $t=t_n$ and $x_k^n$ is the numerical approximation of the $k$th component of $x$ at $t=t_n$.

Restricting ourselves to two-step methods, we define the PDG methods as follows.

\begin{definition}
For the energy $H$ of \eqref{original eq}, consider the polarised energy as a function $\tilde{H} : \mathbb{R}^\di \times \mathbb{R}^\di \rightarrow \mathbb{R}$ satisfying the properties
\begin{align*}
\tilde{H}(x,x) &= H(x),\\
\tilde{H}(x,y) &= \tilde{H}(y,x).
\end{align*}
A polarised discrete gradient (PDG) for $\tilde{H}$ is a function $\overline{\nabla}\tilde{H} : \mathbb{R}^\di \times \mathbb{R}^\di \times \mathbb{R}^\di \rightarrow \mathbb{R}^\di$ satisfying
\begin{align}
\tilde{H}(y,z) - \tilde{H}(x,y) &= 
\frac{1}{2}(z-x)^T \overline{\nabla}\tilde{H}(x,y,z), \label{eq:dg1}\\
\overline{\nabla}\tilde{H}(x,x,x)&= \nabla H(x), \nonumber
\end{align}
and the corresponding polarised discrete gradient scheme is given by
\begin{equation}\label{eq:polscheme}
\frac{x^{n+2}-x^{n}}{2 \Delta t} = S \overline{\nabla}\tilde{H}(x^n,x^{n+1},x^{n+2}).
\end{equation}
\end{definition}

\begin{prop}
The numerical scheme \eqref{eq:polscheme} preserves the polarised invariant $\tilde{H}$ in the sense that $\tilde{H}(x^{n},x^{n+1}) = \tilde{H}(x^0,x^1)$ for all $n\geq 0$.
\end{prop}
\begin{proof}
\begin{align*}
\tilde{H}(x^{n+1},x^{n+2})-\tilde{H}(x^n,x^{n+1})
&= \frac{1}{2}(x^{n+2}-x^n)^T \overline{\nabla}\tilde{H}(x^n,x^{n+1},x^{n+2}) \\
&= \Delta t \overline{\nabla}\tilde{H}(x^n,x^{n+1},x^{n+2})^T S \overline{\nabla}\tilde{H}(x^n,x^{n+1},x^{n+2}) \\
&= 0,
\end{align*}
where the last equality follows from the skew-symmetry of $S$.
\end{proof}
We remark here that in the cases where we seek a time-stepping scheme for the system of Hamiltonian ODEs resulting from discretizing a Hamiltonian PDE in space in an appropriate manner, e.g. as described in \cite{celledoni2012preserving}, $H$ will be a discrete approximation to an integral $\mathcal{H}$. Thus a two-step PDG method and a standard one-step discrete gradient method, the latter in general fully implicit, will preserve two different discrete approximations separately to the same $\mathcal{H}$.

The task of finding a PDG satisfying \eqref{eq:dg1} is approached differently in our two main references, \cite{matsuo2001dissipative, matsuo2002spatially, matsuo2007new, furihata2011discrete} and \cite{dahlby2011general}. Furihata, Matsuo and coauthors apply a generalisation of the approach introduced by Furihata in \cite{furihata1999finite} for finding discrete variational derivatives, while Dahlby and Owren suggest a generalisation of the average vector field (AVF) discrete gradient \cite{mclachlan1999geometric}, given by
\begin{equation*}
\overline{\nabla}_\text{AVF}\tilde{H}(x,y,z) = 2 \int_0^1 \nabla_x \tilde{H}(\xi x + (1-\xi)z,y) \, \mathrm{d}\xi,
\end{equation*}
where $\nabla_x \tilde{H}(x,y)$ is the gradient of $\tilde{H}(x,y)$ with respect to its first argument.
Provided that the spatial discretization is performed in the same way, these two approaches lead to the same scheme for an $\tilde{H}$ quadratic in each of its arguments, as does a generalisation of the midpoint discrete gradient of Gonzalez \cite{gonzalez1996time}. Based on this, we present the most straightforward approach for finding this specific PDG for the cases we are studying in this paper:

\begin{prop}\label{pdg}
Given an $\tilde{H}(x,y)$ that is at most quadratic in each of its arguments, define $\nabla_x \tilde{H}(x,y)$ as the gradient of $\tilde{H}$ with respect to its first argument. Then a PDG for $\tilde{H}$ is given by
\begin{equation}\label{eq:PDG}
\overline{\nabla}\tilde{H}(x,y,z) = 2 \nabla_x\tilde{H}(\frac{x+z}{2},y).
\end{equation}
\end{prop}
\begin{proof}
We may write
\begin{equation*}
\tilde{H}(x,y) = x^T A(y) x + b(y)^T x + c(y),
\end{equation*}
for some symmetric $A : \mathbb{R}^\di \rightarrow \mathbb{R}^\di \times \mathbb{R}^\di$, $b : \mathbb{R}^\di \rightarrow \mathbb{R}^\di$ and $c : \mathbb{R}^\di \rightarrow \mathbb{R}$. Then
\begin{equation*}
\nabla_x \tilde{H}(x,y) = 2 A(y) x + b(y),
\end{equation*}
and
\begin{align*}
\nabla_x \tilde{H}(\frac{x+z}{2},y)^T (z-x) &= (2 A(y) \frac{x+z}{2} + b(y))^T (z-x) \\
&= z^T A(y) z +b(y)^T z - x^T A(y) x - b(y)^T x \\
&= \tilde{H}(y,z) - \tilde{H}(x,y).
\end{align*}
Furthermore,
\begin{equation*}
\overline{\nabla}\tilde{H}(x,x,x) = 2 \nabla_x\tilde{H}(x,x) = \nabla H(x).
\end{equation*}
\end{proof}
As remarked in Theorem $4.5$ of \cite{dahlby2011general}: if the polarised energy $\tilde{H}(x,y)$ is at most quadratic in each of its arguments, the scheme \eqref{eq:polscheme} with the PDG \eqref{eq:PDG} is linearly implicit.

An alternative to \eqref{eq:PDG} could be a generalisation of the Itoh--Abe discrete gradient \cite{itoh1988hamiltonian}, defined by its $i$-th component
\begin{align*}
\overline{\nabla}_\text{IA}\tilde{H}(x,y,z)_i  &= 2
  \begin{cases}
    \bar{\partial}\tilde{H}(x,y,z)_i       & \quad \text{if } x_i \neq z_i,\\
    \frac{\partial \tilde{H}}{\partial x_i}((z_1,\ldots,z_{i-1},x_{i},\ldots,x_\di),y)  & \quad \text{if } x_i = z_i,
  \end{cases}
\end{align*}
where
$$\bar{\partial}\tilde{H}(x,y,z)_i  = \dfrac{\tilde{H}((z_1,\ldots,z_i,x_{i+1},\ldots,x_\di),y) - \tilde{H}((z_1,\ldots,z_{i-1},x_{i},\ldots,x_\di),y)}{z_i-x_i}.$$
A symmetrized variant of this, given by $\overline{\nabla}_\text{SIA}\tilde{H}(x,y,z) := \frac{1}{2}(\overline{\nabla}_\text{IA}\tilde{H}(x,y,z) + \overline{\nabla}_\text{IA}\tilde{H}(z,y,x))$ is again identical to \eqref{eq:PDG}, whenever $\tilde{H}$ is quadratic in each of its arguments.

\subsection{A general framework and Kahan's method}
For ODEs of the form \eqref{original eq}, consider the two-step schemes of the form
\begin{equation}\label{eq:expanded_form}
\frac{x^{n+2}-x^n}{2\Delta t}=S\sum_{i,j=1}^3 \alpha_{ij}(H{''}(x^{n-1+i})x^{n-1+j}+\beta(x^{n-1+i})),
\end{equation}
where $H{''} : \mathbb{R}^\di \rightarrow \mathbb{R}^\di \times \mathbb{R}^\di$ is the Hessian matrix of $H$ and  $\beta(x):=2 \nabla H(x) - H{''}(x)x$. 
For cubic $H$, this scheme is linearly implicit if and only if $\alpha_{33} = 0$.

In this section, we first consider the case when the Hamiltonian is a cubic homogeneous polynomial, in which case the term $\beta(x)$ in \eqref{eq:expanded_form} will disappear, and then generalise the results to the non-homogeneous case.

\begin{theorem}\label{theorem_kahan_general}
The scheme \eqref{eq:expanded_form} with $\alpha_{2 1} = \alpha_{2 3} = \frac{1}{4}$, $\alpha_{i j} = 0$ otherwise, i.e.
\begin{equation}\label{eq:kahan2s}
\frac{x^{n+2}-x^n}{2\Delta t} = \frac{1}{4}S H''(x^{n+1})(x^n+x^{n+2}),
\end{equation}
%{\color{blue} and the first step computed by Kahan's method is equivalent with the two-step method that performing Kahan's method two steps,}  when applied to ODEs of the form \eqref{original eq} with homogeneous cubic $H$.
where $x^1$ is found from $x^0$ by Kahan's method, is equivalent to Kahan's method over two consecutive steps, when applied to ODEs of the form \eqref{original eq} with homogeneous cubic $H$.
\end{theorem}
\begin{proof}
As shown in \cite{celledoni2012geometric}, Kahan's method can be written into a Runge--Kutta form
\begin{align}\label{one step Kahan}
 \frac{x^{n+1}-x^n}{\Delta t} = -\frac{1}{2}f(x^{n})+2f(\frac{x^{n}+ x^{n+1}}{2})-\frac{1}{2}f(x^{n+1}).
\end{align}
Two steps of this can be written as
\begin{equation}\label{Two_step_kahan}
\frac{x^{n+2}-x^n}{2\Delta t}=-\frac{1}{4}f(x^{n})-\frac{1}{2}f(x^{n+1})-\frac{1}{4}f(x^{n+2})+f(\frac{x^{n}+ x^{n+1}}{2})+f(\frac{x^{n+1}+ x^{n+2}}{2}).
\end{equation}
Using that for a homogeneous cubic $H$ we have $\nabla H(x) = \frac{1}{2}H''(x) x$, $H''(x) y = H''(y) x$ and $H''(x+y)=H''(x)+H''(y)$, and inserting $f(x) = S \nabla H(x)$ in \eqref{Two_step_kahan}, we get \eqref{eq:kahan2s}.
On the other hand, if we have found $x^{n+1}$ by Kahan's method and $x^{n+2}$ by \eqref{Two_step_kahan}, we see that by subtracting \eqref{one step Kahan} from \eqref{Two_step_kahan} we get \eqref{one step Kahan} with $n$ replaced by $n+1$.
%{\color{blue}  On the other hand, according to $ H''(x)y = \nabla H(x+y)- \nabla H(x)- \nabla H(y)$, equation \eqref{eq:kahan2s} can be rewritten as \eqref{Two_step_kahan}. Since the first step is given by the Kahan's method \eqref{one step Kahan}, substracting \eqref{one step Kahan} from \eqref{Two_step_kahan}, we can get Kahan's scheme from the $n+1$ step to $n+2$ step, i.e. equation \eqref{one step Kahan} with $n$ replcaed by $n+1$.}
\end{proof}

\begin{remark}
The scheme \eqref{eq:kahan2s} with the first step computed by Kahan's method preserves the polarised invariant $\tilde{H}(x^n,x^{n+1})= \frac{1}{6}(x^n)^T H''(\frac{x^n+x^{n+1}}{2})x^{n+1}$, since Kahan's method preserves this polarised invariant \cite{celledoni2012geometric}. We note that the scheme \eqref{eq:kahan2s} satisfies
$$(x^n)^T H''(x^n) x^{n+1} = (x^{n+1})^T H''(x^{n+2}) x^{n+2},$$
independent of how $x^1$ is found, following from the skew symmetry of the matrix $S$.  However, it preserves the polarised invariant $\frac{1}{6}(x^n)^T H''(\frac{x^n+x^{n+1}}{2})x^{n+1}$ only if Kahan's method or an equivalent scheme is used to calculate $x^1$ from $x^0$.
\end{remark}

%Kahan's method preserves the polarised invariant $\tilde{H}(x,y) = \frac{1}{3} \nabla H(x)y = \frac{1}{3} \nabla H(y)x = \frac{1}{6}x^T H''(\frac{x+y}{2})y$ \cite{celledoni2012geometric}.

%It can be shown that many well known Runge--Kutta methods composed with itself over two consecutive steps is a method in the class \eqref{eq:expanded_form} when applied to \eqref{original eq} with cubic $H$.
%As two examples, the implicit midpoint method over two steps is \eqref{eq:expanded_form} with $\alpha_{1 1} = \alpha_{3 3} = \frac{1}{16}, \alpha_{2 1} = \alpha_{2 2} = \alpha_{2 3} = \frac{1}{8}$, $\alpha_{i j} = 0$ otherwise, while the trapezoidal rule is \eqref{eq:expanded_form} with $\alpha_{1 1} = \alpha_{3 3} = \frac{1}{8}, \alpha_{2 2} = \frac{1}{4}$, $\alpha_{i j} = 0$ otherwise. The integral-preserving average vector field method \cite{quispel08anc} over two steps is \eqref{eq:expanded_form} with $\alpha_{1 1} = \alpha_{2 1} = \alpha_{2 3} = \alpha_{3 3} = \frac{1}{12}, \alpha_{2 2} = \frac{1}{6}$, $\alpha_{i j} = 0$ otherwise.

A special case of the PDG method which preserves the same polarised Hamiltonian as Kahan's method, can also be written on the form \eqref{eq:expanded_form}:
\begin{theorem}\label{theorem_pdg_general}
For a homogeneous cubic $H$ and the polarised energy given by
$$\tilde{H}(x,y) = \frac{1}{6}x^T H''(\frac{x+y}{2})y,$$
the scheme \eqref{eq:polscheme} with the PDG \eqref{eq:PDG} applied to \eqref{original eq} is equivalent to \eqref{eq:expanded_form} with $\alpha_{2 1} = \alpha_{2 2} = \alpha_{2 3} = \frac{1}{6}$, $\alpha_{i j} = 0$ otherwise, i.e.
\begin{align}\label{PDg-two step}
\frac{x^{n+2}-x^n}{2\Delta t} &= \frac{1}{6}S H''(x^{n+1})(x^n+x^{n+1}+x^{n+2}).
\end{align}
\end{theorem}

\begin{proof}
\begin{equation*}
\nabla_x \tilde{H}(x,y) = \frac{1}{6}H''(\frac{x+y}{2})y + \frac{1}{6}H''(\frac{y}{2})x = \frac{1}{12}H''(2 x+y)y,
\end{equation*}
and thus
\begin{equation*}
\overline{\nabla} \tilde{H}(x,y,z) = 2 \nabla_x \tilde{H}(\frac{x+z}{2},y) = \frac{1}{6}H''(x+y+z)y = \frac{1}{6}H''(y)(x+y+z).
\end{equation*}
\end{proof}

It can be shown that many well known Runge--Kutta methods performed over two consecutive steps are methods in the class \eqref{eq:expanded_form} when applied to \eqref{original eq} with $H$ cubic.
As two examples, the implicit midpoint method over two steps is \eqref{eq:expanded_form} with $\alpha_{1 1} = \alpha_{3 3} = \frac{1}{16}, \alpha_{2 1} = \alpha_{2 2} = \alpha_{2 3} = \frac{1}{8}$, $\alpha_{i j} = 0$ otherwise, while the trapezoidal rule is \eqref{eq:expanded_form} with $\alpha_{1 1} = \alpha_{3 3} = \frac{1}{8}, \alpha_{2 2} = \frac{1}{4}$, $\alpha_{i j} = 0$ otherwise. The integral-preserving average vector field method \cite{quispel08anc} over two steps is \eqref{eq:expanded_form} with $\alpha_{1 1} = \alpha_{2 1} = \alpha_{2 3} = \alpha_{3 3} = \frac{1}{12}, \alpha_{2 2} = \frac{1}{6}$, $\alpha_{i j} = 0$ otherwise.

%{\color{blue} 
%\begin{remark}
%%It is obvious that the scheme \eqref{eq:expanded_form} satisfies that $$\sum_{i,j=1}^3 \alpha_{ij}(x^{k+2})^TH''(x^{k-1+i})x^{k-1+j}=\sum_{i,j=1}^3 \alpha_{ij}(x^{k})^TH''(x^{k-1+i})x^{k-1+j}, $$
%%by the skew symmetry of the matrix $S$. 
%Inspired by Theorem \ref{theorem_pdg_general}, we have tried to find different parameters for \eqref{eq:expanded_form} so that the scheme is linearly implicit as well as preserving a polarised energy independent of the relation between the two starting points. We have found that the scheme \eqref{PDg-two step} is the only choice.
%\end{remark}
%} 

Now, in the cases where $H$ is non-homogeneous, one can use the technique employed in \cite{celledoni2012geometric}, i.e. adding one variable $x_0$ to generate an equivalent problem to the original one, for a homogeneous Hamiltonian $\bar{H} : \mathbb{R}^{\di+1} \rightarrow \mathbb{R}$ defined such that $\bar{H}(1,x_1,\ldots,x_\di) = H(x_1,\ldots,x_\di)$. Also constructing the $(\di+1) \times (\di+1)$ skew-symmetric matrix $\bar{S}$ by adding a zero initial row and a zero initial column to $S$, we get that solving the system
\begin{equation}
\begin{aligned}
\dot{\bar{x}}&=\bar{S}\nabla \bar{H}(\bar{x}),\quad \bar{x}\in \mathbb{R}^{\di+1}\\
\bar{x}(0)&=(1,x^{0}),
\end{aligned}
\label{eq:homogeneoussystem}
\end{equation}
is equivalent to solving \eqref{original eq}. Following the above results for the homogeneous $\bar{H}$ and \eqref{eq:homogeneoussystem}, we can generalise Theorem \ref{theorem_kahan_general} and Theorem \ref{theorem_pdg_general} for all cubic $H$. Generalisations of the preservation properties follow directly; e.g., Kahan's method and the PDG method can preserve the perturbed energy $\tilde{H}(x^n,x^{n+1}):=\frac{1}{6}(\bar{x}^n)^T\bar{H}''(\frac{\bar{x}^n+\bar{x}^{n+1}}{2})\bar{x}^{n+1}$ also for non-homogeneous cubic $H$.

\section {Numerical experiments}\label{Numerical examples}
To have a better understanding of the above methods, we will apply them to systems of two different PDEs: the Korteweg--de Vries (KdV) equation and the Camassa--Holm equation. We will compare our methods to the midpoint method, which is a symplectic, fully implicit method. We solve the two PDEs by discretizing in space to obtain a Hamiltonian ODE system of the type \eqref{original eq} and then applying the PDG method (denoted by PDGM), Kahan's method (Kahan) and the midpoint method (MP) to this.

\subsection {Camassa--Holm equation}
In this section, we consider the Camassa--Holm equation
\begin{equation*}%\label{eq_CH}
u_t - u_{xxt} + 3 u u_x = 2 u_x u_{xx} + u u_{xxx}
\end{equation*}
defined on the periodic domain $ \Sb := \RR / L \mathbb{Z} $.  
It has the conserved quantities
\begin{align*}
\mathcal{H}_1 \left[u \right] = \frac{1}{2} \int_{\Sb} (u^2 + u_x^2) \, \rd x,\qquad 
\mathcal{H}_2 \left[u \right] = \frac{1}{2} \int_{\Sb} \left( u^3 + u u_x^2\right) \rd x. 
\end{align*}
Here we consider the variational form of the Hamiltonian $\mathcal{H}_2$:
\begin{align}
 (1-\partial_x^2) u_t = - \partial_x \frac{\delta \mathcal{H}_2}{\delta u }, \qquad \frac{\delta \mathcal{H}_2}{\delta u } = \frac{3}{2} u^2 + \frac{1}{2} u_x^2 - ( u u_x)_x . \label{eq_CH_variation}
\end{align}

We follow the approach presented in \cite{celledoni2012preserving} and semi-discretize the energy $ \mathcal{H}_2 $ of \eqref{eq_CH_variation} as
\begin{equation}\label{eq_discrete_energy}
H_2 (u) \Delta x= \frac{1}{2} \sum_{k=1}^K \left( u_k^3 + u_k \frac{(\fd_x u_k)^2 + (\bd_x u_k)^2}{2} \right) \Delta x, 
\end{equation}
where the difference operators $\fd_x$ and $\bd_x$ are defined by
\begin{equation*}
\fd_x u_k := \frac{u_{k+1} - u_k}{\Delta x}, \qquad \bd_x u_k := \frac{u_{k} - u_{k-1}}{\Delta x}.
\end{equation*}
For later use, we here also introduce the notation
\begin{equation*}
\cd_x u_k := \frac{u_{k+1} - u_{k-1}}{2\Delta x}, \qquad \cd[2]_x u_k := \frac{u_{k+1} - 2 u_k + u_{k-1}}{(\Delta x)^2},
\end{equation*}
\begin{equation*}
\fa_x u_k := \frac{u_{k+1} + u_k}{2}, \qquad \ba_x u_k := \frac{u_{k} +u_{k-1}}{2},
\end{equation*}
and the matrices corresponding to the difference operators $\fd_x$, $\bd_x$, $\cd_x$, $\cd[2]_x$, $\fa_x$ and $\ba_x$, which are denoted by $D^+$, $D^-$, $\Dc$, $\Dc[2]$, $M^+$ and $M^-$. Denoting the numerical solution $U=[u_1,\ldots,u_K]^T$, and by using the properties of the above difference operators, we thus get 
\[  \nabla H_2 (U)  =\frac{3}{2}U^2_{\boldsymbol{\cdot}}+\frac{1}{2}M^-(D^+U)_{\boldsymbol{\cdot}}^2-\frac{1}{2}\Dc[2]U_{\boldsymbol{\cdot}}^2,\]
where $U_{\boldsymbol{\cdot}}^2$ is the elementwise square of $U$. Then the semi-discretized system for the Camassa--Holm equation becomes
\begin{align}\label{semi-ode-CHeq}
\dot{U}=S\nabla H_2(U)
          = -(I-\Dc[2])^{-1}\Dc\nabla H_2(U).
\end{align}
The above-mentioned schemes applied to \eqref{semi-ode-CHeq} give us
\begin{align}
(I-\Dc[2])\frac{U^{n+1}-U^n}{\Delta t}&=-\Dc\nabla H_2(\frac{U^{n+1}+U^n}{2}), \qquad \text{(MP)}\label{semi-ode-CH_Midpoint}\\
(I-\Dc[2])\frac{U^{n+1}-U^n}{\Delta t}&=-\frac{1}{2}\Dc H_2^{''}(U^{n})U^{n+1}, \qquad \text{(Kahan)}\label{semi-ode-CH_Kahan_two}\\
(I-\Dc[2])\frac{U^{n+2}-U^n}{2\Delta t}&=-\Dc \overline{\nabla} \tilde{H}_2 (U^{n},U^{n+1},U^{n+2}), \qquad \text{(PDGM)}\label{discreteE-CH_polar}
\end{align}
where $H_2^{''}(U)=3\,\text{diag}(U)+M^-\,\text{diag}(D^+U)D^+- \Dc[2]\,\text{diag}(U)$ is the Hessian of $H_2(U)$ and\\
$\overline{\nabla} \tilde{H}_2 (U^{n},U^{n+1},U^{n+2})$ is the PDG of Proposition \ref{pdg} with polarised discrete energy
\begin{equation*}\label{multistep_energy}
\begin{split}
\tilde{H}_2 (U^n,U^{n+1}) \Delta x := &\frac{1}{2} \sum_{k=1}^K\bigg( u_k^{n} u^{n+1}_k \frac{u_k^{n}+ u_k^{n+1}}{2} + a  (\fa_x \frac{u_k^{n} + u_k^{n+1}}{2}) ( \fd_x u_k^{n} ) (\fd_x u_k^{n+1}) \\
& + (1-a) \frac{ (\fa_x u_k^{n}) (\fd_x u_k^{n+1})^2 + (\fa_x u_k^{n+1}) (\fd_x u_k^{n})^2 }{2}\bigg) \Delta x,
\end{split}
\end{equation*}
for some $a \in \mathbb{R}$, typically between $-1$ and $2$.
\begin{remark}
We performed numerical experiments for finding a good choice of the parameter $a$ in PDGM \eqref{discreteE-CH_polar} and based on these set $a=\frac{1}{2}$ in the following.
\end{remark}

\subsubsection {Numerical tests for the Camassa--Holm equation}

\noindent \textbf{Example 1 (Single peakon solution):} In this numerical test, we consider the same experiment as in \cite{cohen2008multi}, where multisymplectic schemes are considered for the Camassa--Holm equation with
$$u(x,0)=\frac{\cosh(\lvert x-\frac{L}{2}\rvert-\frac{L}{2})}{\cosh(L/2)},$$
$x\in[0,L]$, $L=40$, $t\in[0,T]$, $T=5$, spatial step size $\Delta x=0.04$ and time step size $\Delta t=0.0002$. All our methods keep a shape close to the exact solution except some small oscillatory tails, also observed in \cite{cohen2008multi}, resulting from the semi-discretization, see Figure \ref{ener_solution} (the right two plots). 
The numerical simulations show that the global error is mainly due to the shape error\footnote{\label{myfootnote} Shape error is defined by $\epsilon_{\text{shape}}:=\underset{\tau}{\text{min}}\parallel U^n-u(\cdot-\tau)\parallel_2^2$, and phase error is defined by 
$\epsilon_{\text{phase}}:=\lvert\underset{\tau}{\text{argmin}}\parallel U^n-u(\cdot-\tau)\parallel_2^2-ct_n\rvert$, \cite{dahlby2011general}.},
see Figure \ref{shape-phase-global error}. In Figure \ref{ener_solution} (the left plot), we can see that the numerical energy for all the methods oscillate, but it appears to be bounded. Here we consider also coarser grids. We observe that there appear some small wiggles for both PDGM and Kahan's method for $\Delta t=0.02$ and long time integration $T=100$. However, the wiggles in the solution by PDGM are much more evident than those in the solution of Kahan's method, see Figure \ref{solution-CH-1peakon-Tkahan} (the left two plots). We keep on increasing $\Delta t$ to $0.15$ and $0.2$; we observe that the numerical solution obtained with the PDG method with $\Delta t=0.15$ suffers from evident numerical dispersion, while Kahan's method seems to keep the shape well when comparing to the exact wave. Spurious oscillations appear also in Kahan's method when the time-step is increased to the value $\Delta t=0.2$, see Figure \ref{solution-CH-1peakon-Tkahan} (right).  

\begin{figure}[h!]
\hspace{-10pt}
\centering
      \begin{subfigure}[b]{0.33\textwidth}
      \centering
                \includegraphics[width=0.99\textwidth]{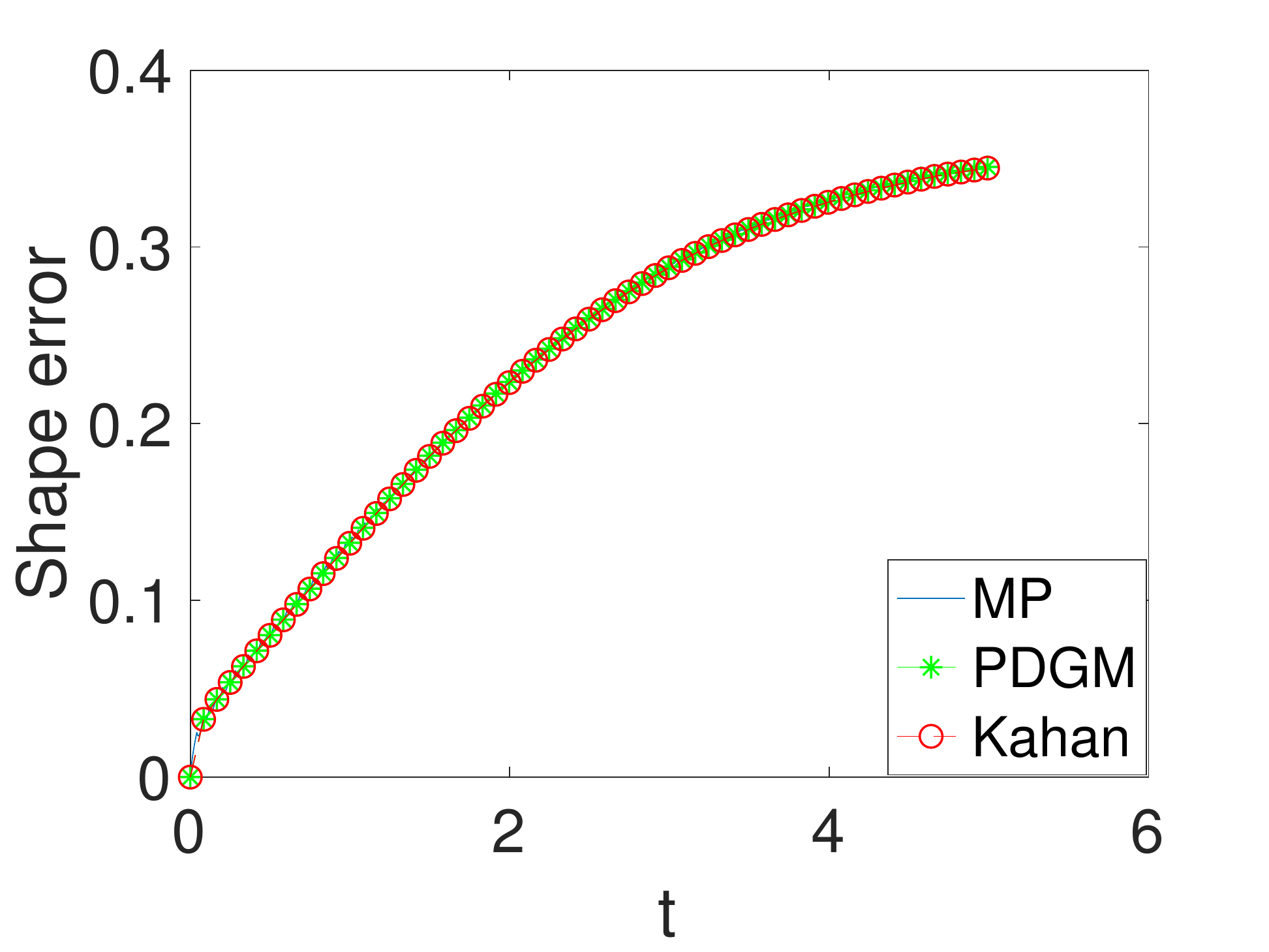}
        \end{subfigure}
        \begin{subfigure}[b]{0.33\textwidth}
        \centering
                \includegraphics[width=0.99\textwidth]{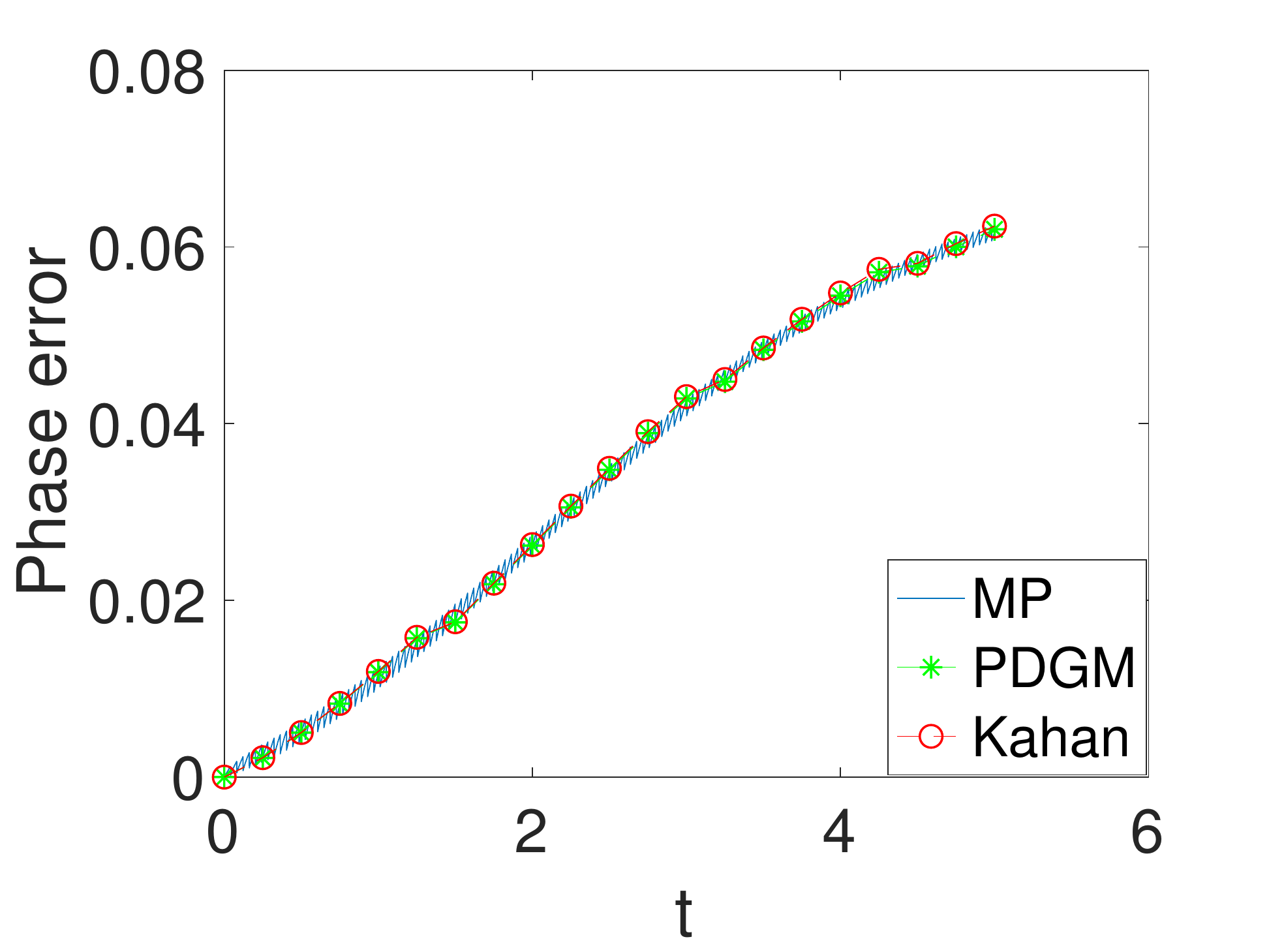}
        \end{subfigure}
 \begin{subfigure}[b]{0.33\textwidth}
      \centering
                \includegraphics[width=0.99\textwidth]{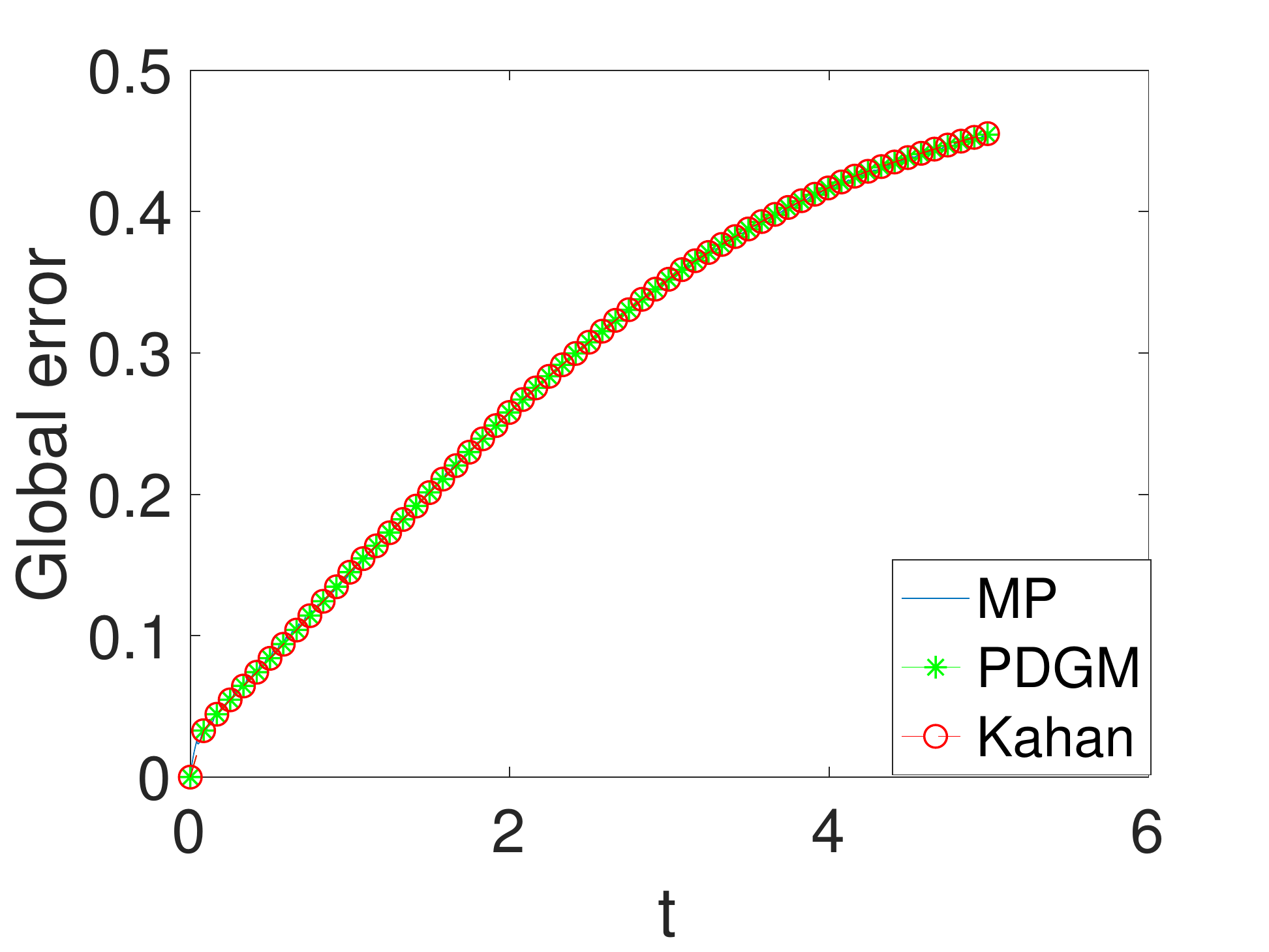}
        \end{subfigure}
      \caption{In this experiment, space step size $\Delta x=0.04$ and time step size $\Delta t=0.0002$. \textbf{Left:} shape error, \textbf{middle:} phase error, \textbf{right:} global error.}\label{shape-phase-global error}
\vspace{20pt}
\end{figure}
\vspace{-40pt}
\begin{figure}[h!]
\hspace{-10pt}
\centering
      \begin{subfigure}[b]{0.33\textwidth}
      \centering
                \includegraphics[width=0.99\textwidth]{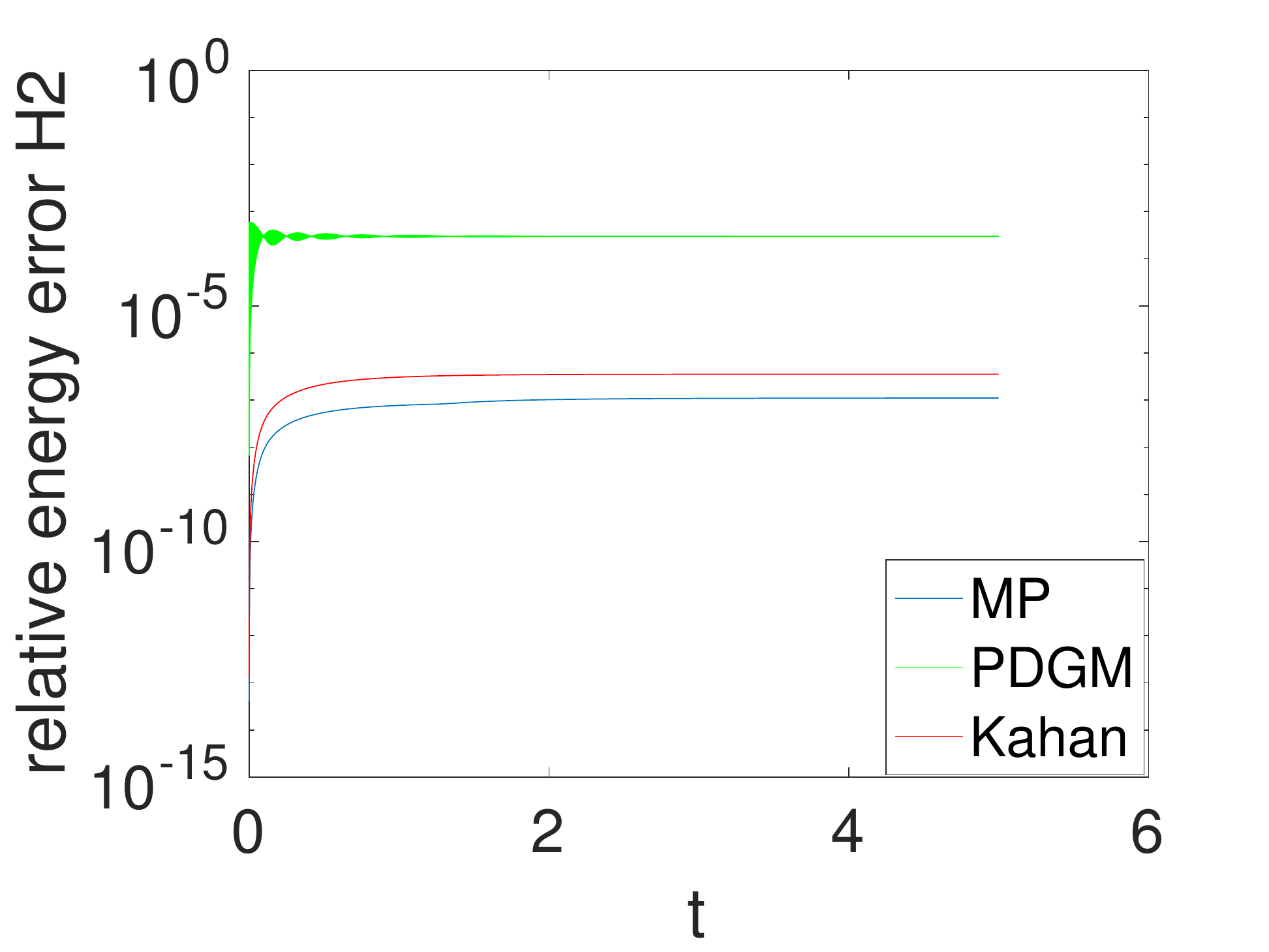}
        \end{subfigure}
        \begin{subfigure}[b]{0.33\textwidth}
        \centering
                \includegraphics[width=0.99\textwidth]{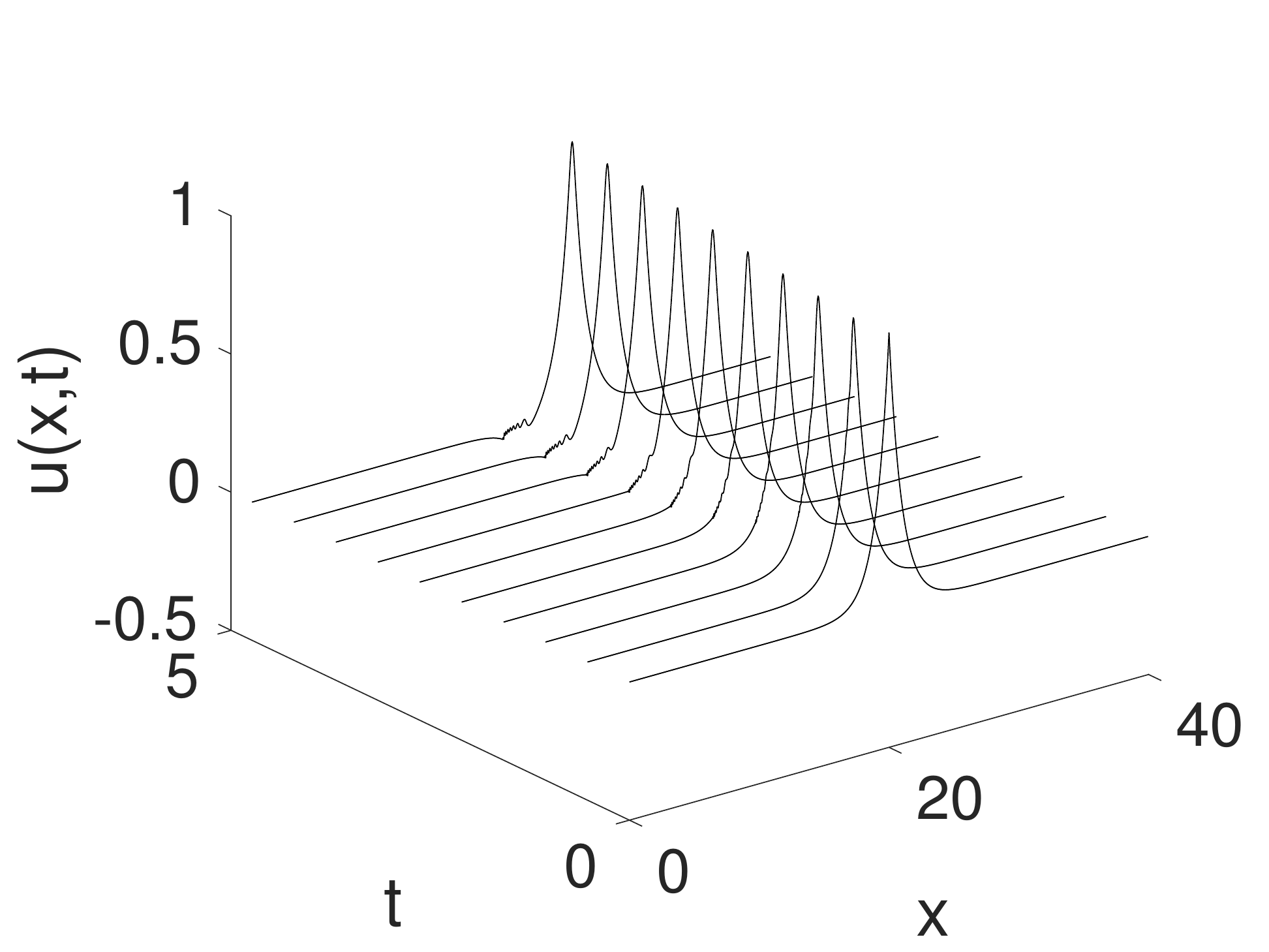}
        \end{subfigure}
 \begin{subfigure}[b]{0.33\textwidth}
      \centering
                \includegraphics[width=0.99\textwidth]{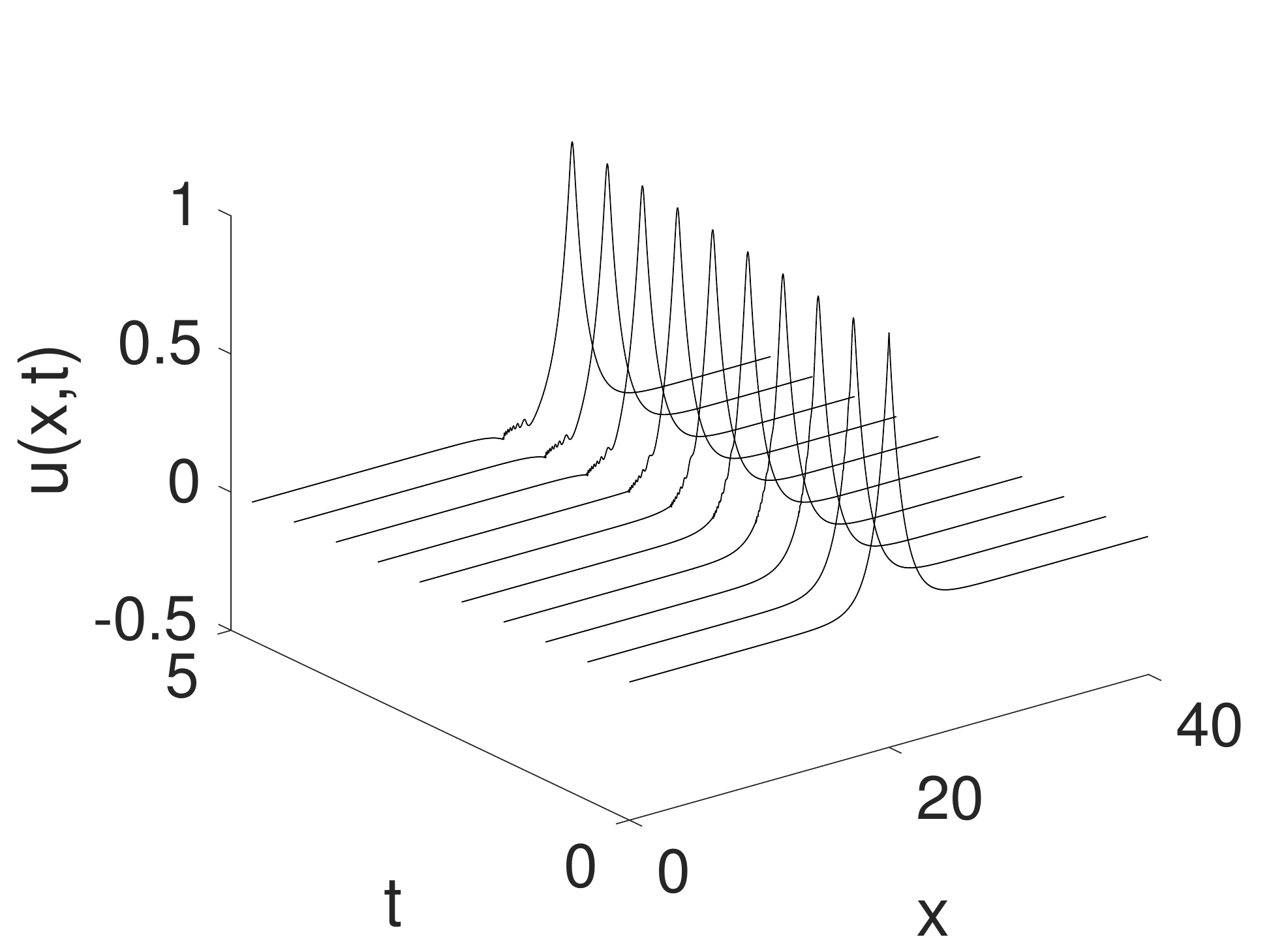}
        \end{subfigure}
      \caption{In this experiment, $\Delta x=0.04$, $\Delta t=0.0002$. \textbf{Left:} relative energy errors.  \textbf{middle:} propagation of the wave by PDGM. \textbf{right:} propagation of the wave by Kahan's method.}\label{ener_solution}
\end{figure}
\begin{figure}[H]
\hspace{-10pt}
\centering
\vspace{10pt}
 \begin{subfigure}[b]{0.33\textwidth}
        \centering
                \includegraphics[width=0.99\textwidth]{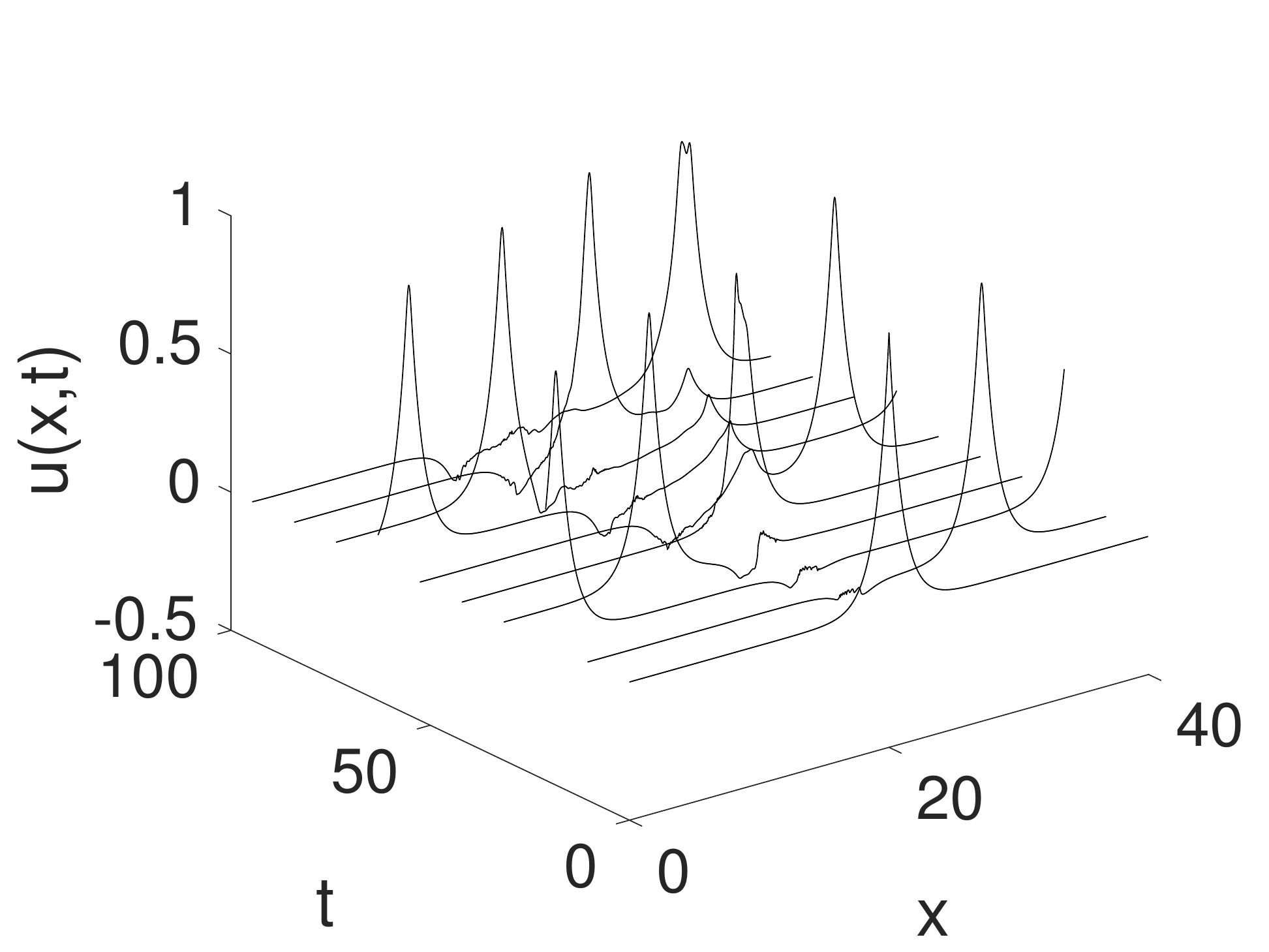}
        \end{subfigure}
        \begin{subfigure}[b]{0.33\textwidth}
        \centering
                \includegraphics[width=0.99\textwidth]{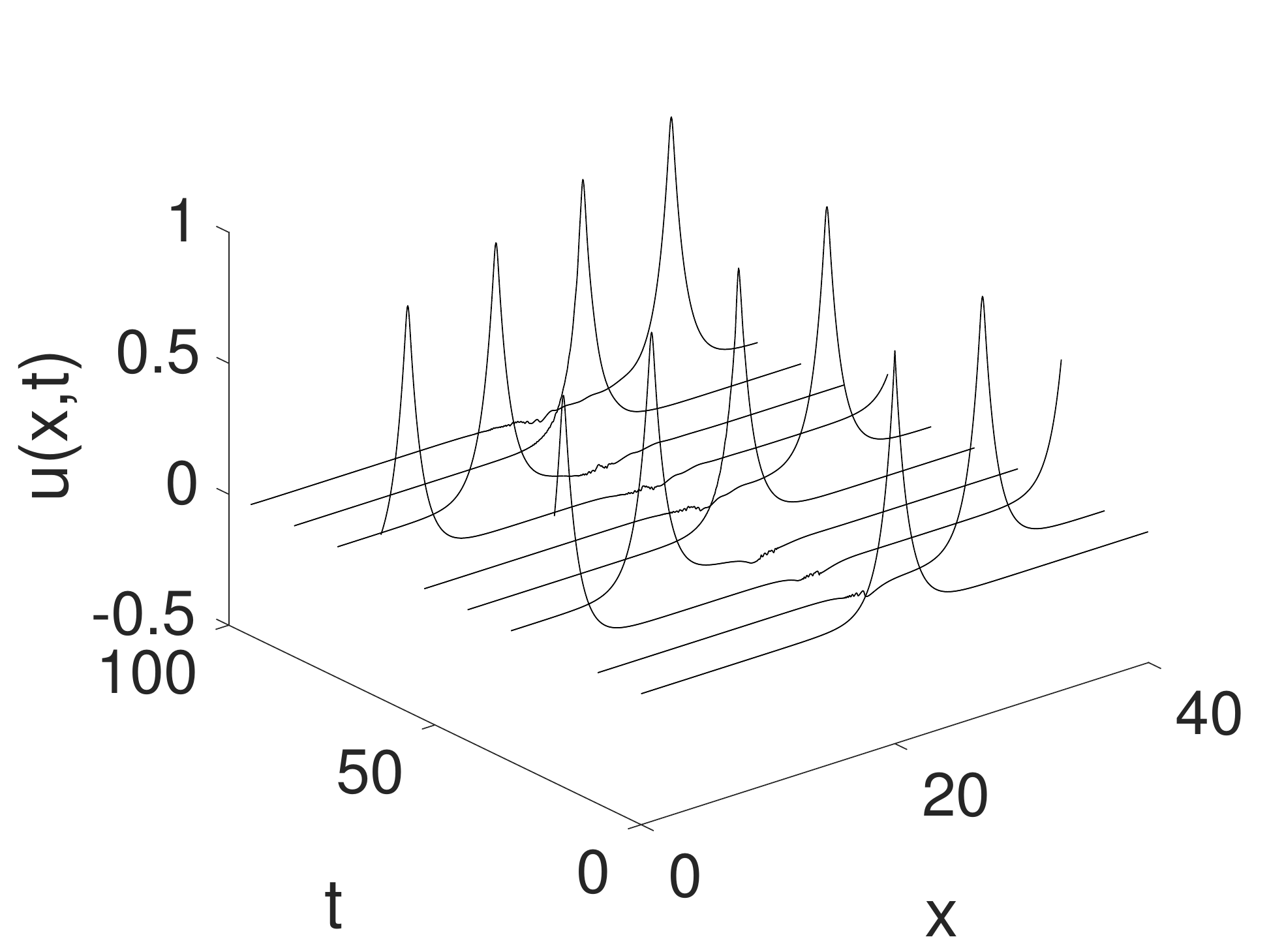}
        \end{subfigure}
 \begin{subfigure}[b]{0.33\textwidth}
      \centering
                \includegraphics[width=0.99\textwidth]{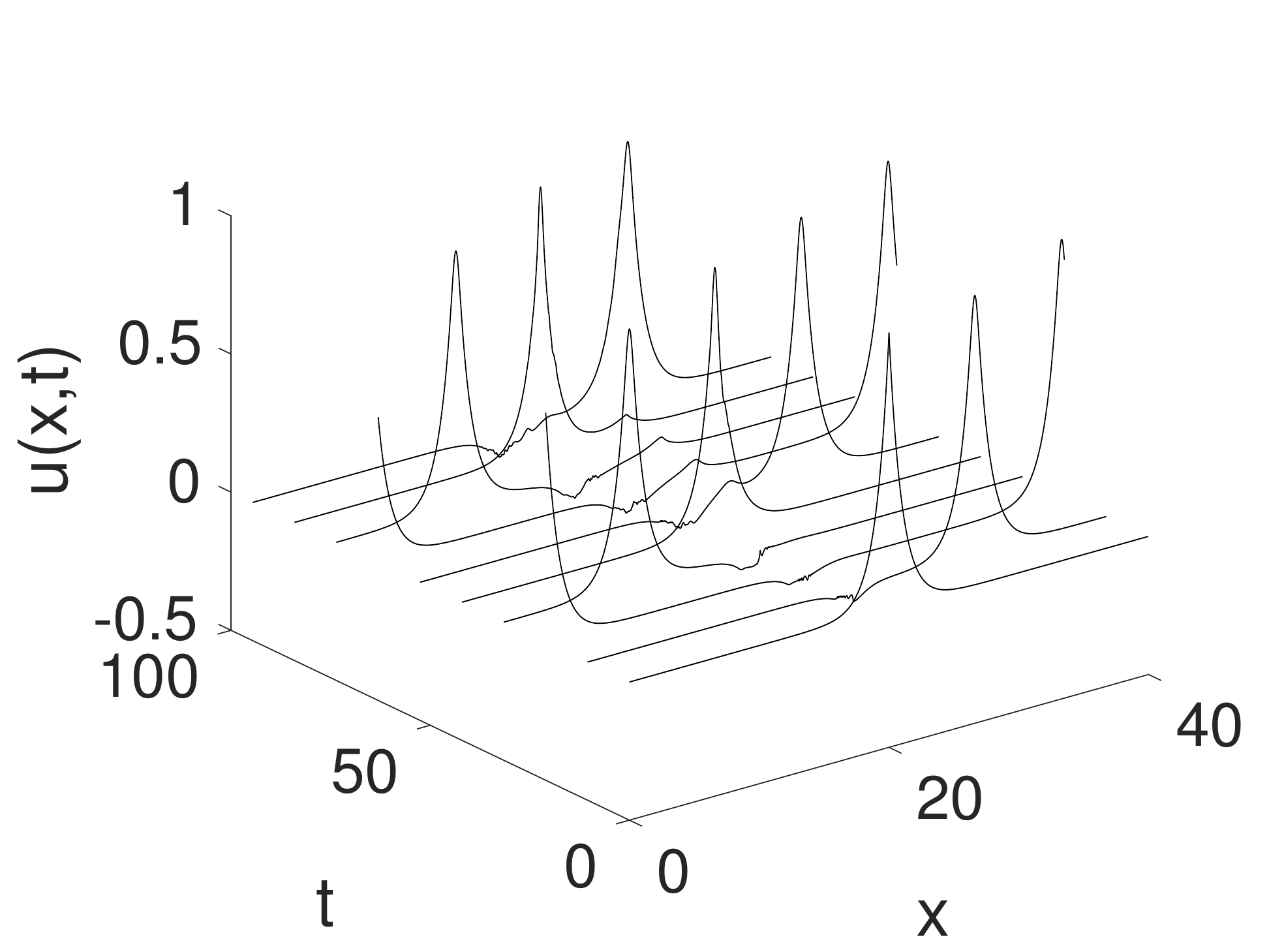}
        \end{subfigure}
      \caption{In this experiment, space step size $\Delta x=0.04$. \textbf{Left:} propagation of the wave by PDGM, $\Delta t=0.02$, \textbf{middle:} propagation of the wave by Kahan's method, $\Delta t=0.02$, \textbf{right:} propagation of the wave by Kahan's method, $\Delta t=0.15$.}\label{solution-CH-1peakon-Tkahan}
\vspace{10pt}
\end{figure}

\noindent \textbf{Example 2 (Two peakons solution):} Now we consider the initial condition
$$u(x,0)=\frac{\cosh(\lvert x-\frac{L}{4}\rvert-\frac{L}{2})}{\cosh(L/2)}+\frac{3}{2}\frac{\cosh(\lvert x-\frac{3L}{4}\rvert-\frac{L}{2})}{\cosh(L/2)},$$
where $x\in[0,L]$, $L=40$, $t\in[0,T]$, $T=5$, and $\Delta x=0.04$, $\Delta t=0.0002$. We observe that all the methods keep the shape of the exact solution very well and the numerical energy appears bounded, see Figure \ref{ener_solution-CH2peakon}.  The numerical simulation shows that the global error is mainly due to the shape error, see Figure \ref{shape-phase-global error-CH2peakon}. When a coarser time grid and longer time integration is considered, $\Delta t=0.02$ and $T=100$, small wiggles appear in the solution of PDGM and Kahan's method, see Figure \ref{solution-CH-2peakon-Tkahan} (the left two figures). We increase $\Delta t$ to $0.2$, and observe that PDGM fails to preserve the shape of the solution, while Kahan's method can still keep a shape close to the exact solution even though also for this method the numerical dispersion increases, see Figure \ref{solution-CH-2peakon-Tkahan} (right).

\begin{figure}[h!]
\hspace{-10pt}
\centering
\vspace{10pt}
      \begin{subfigure}[b]{0.33\textwidth}
      \centering
                \includegraphics[width=0.99\textwidth]{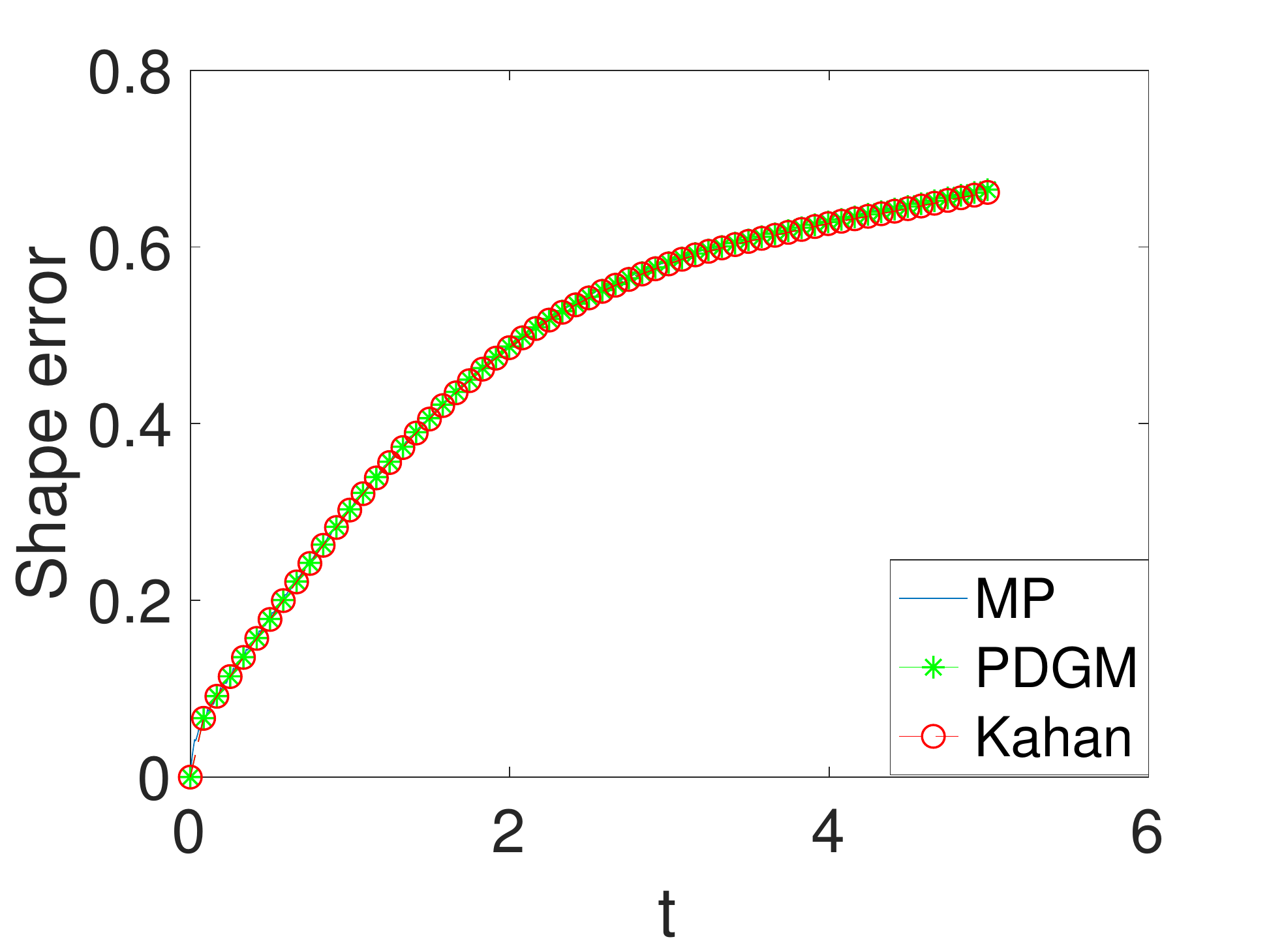}
        \end{subfigure}
        \begin{subfigure}[b]{0.33\textwidth}
        \centering
                \includegraphics[width=0.99\textwidth]{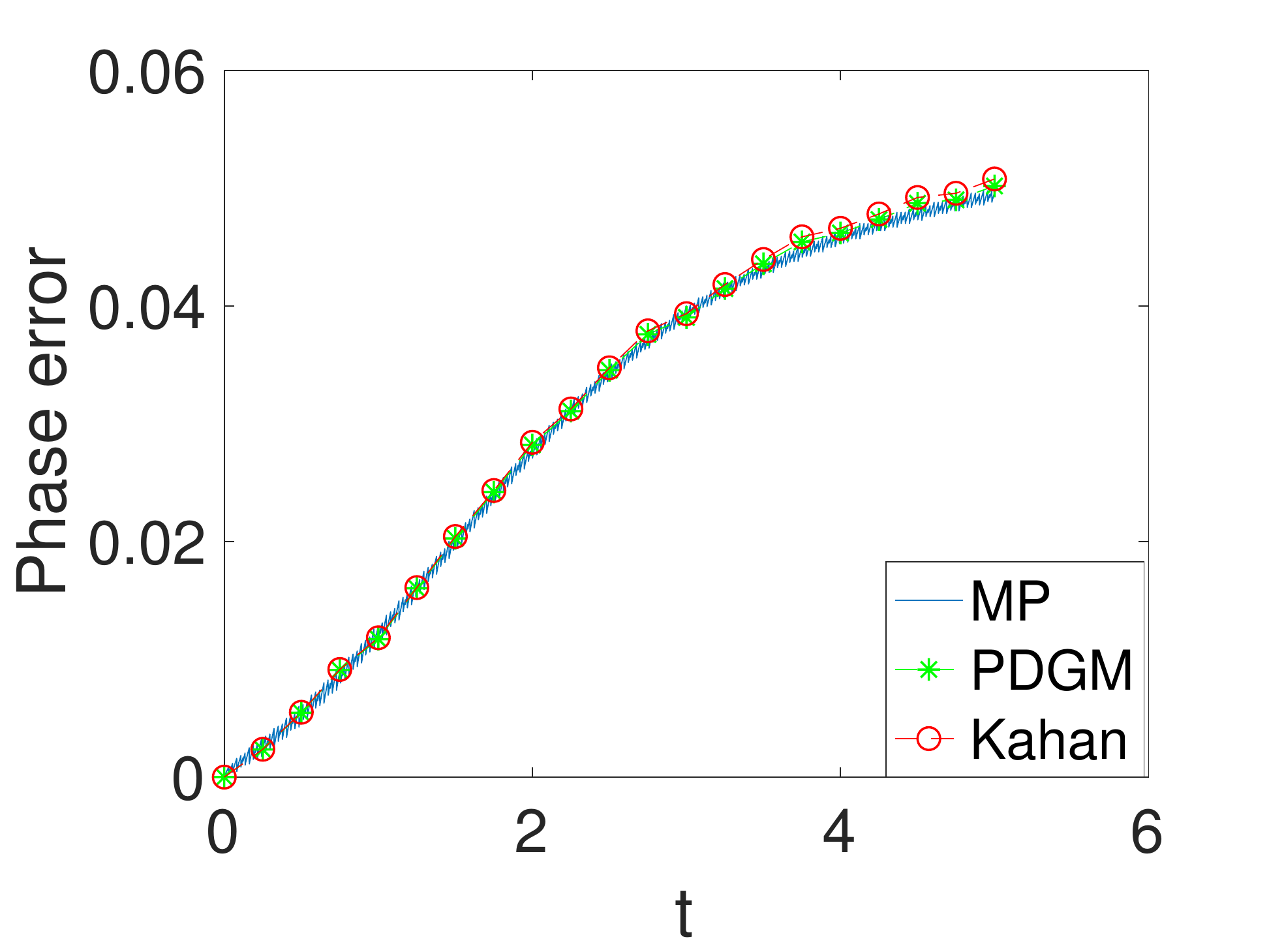}
        \end{subfigure}
 \begin{subfigure}[b]{0.33\textwidth}
      \centering
                \includegraphics[width=0.99\textwidth]{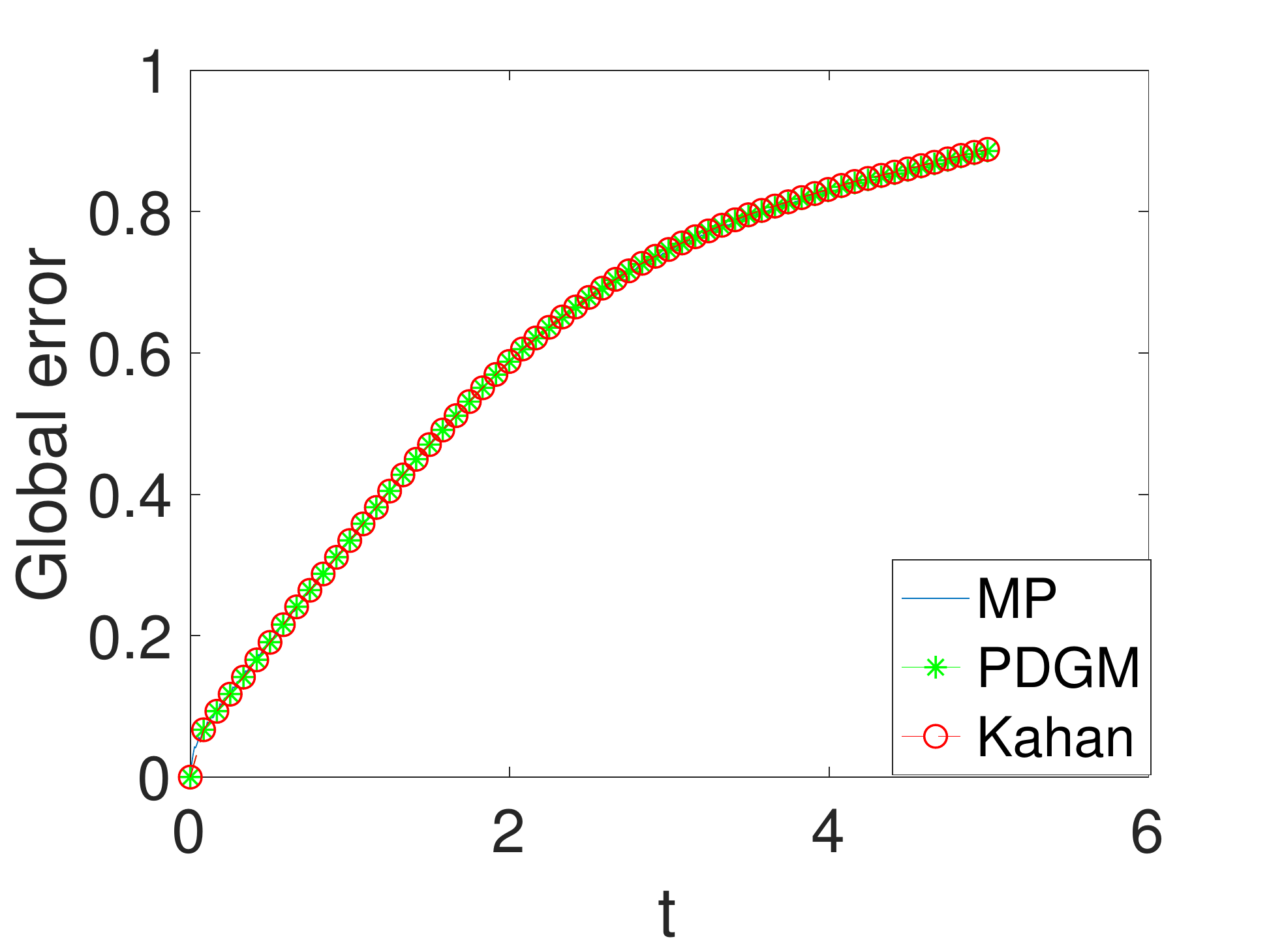}
        \end{subfigure}
      \caption{In this experiment, space step size $\Delta x=0.04$, time step size $\Delta t=0.0002$. \textbf{Left:} shape error, \textbf{middle:} phase error, \textbf{right:} global error.}\label{shape-phase-global error-CH2peakon}
\vspace{20pt}
\end{figure}
\vspace{-40pt}
\begin{figure}[h!]
\hspace{-10pt}
\centering
      \begin{subfigure}[b]{0.33\textwidth}
      \centering
                \includegraphics[width=0.99\textwidth]{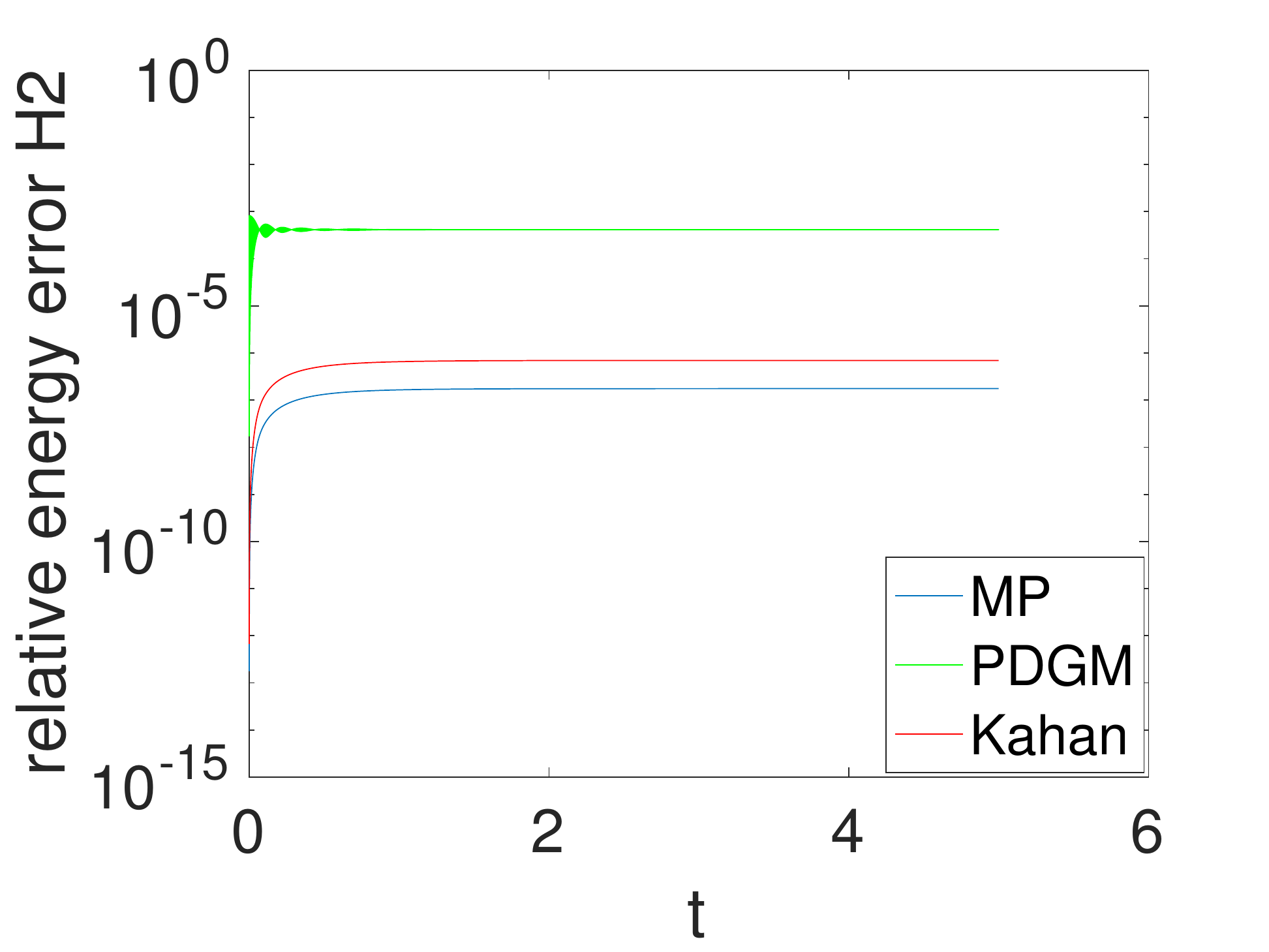}
        \end{subfigure}
        \begin{subfigure}[b]{0.33\textwidth}
        \centering
                \includegraphics[width=0.99\textwidth]{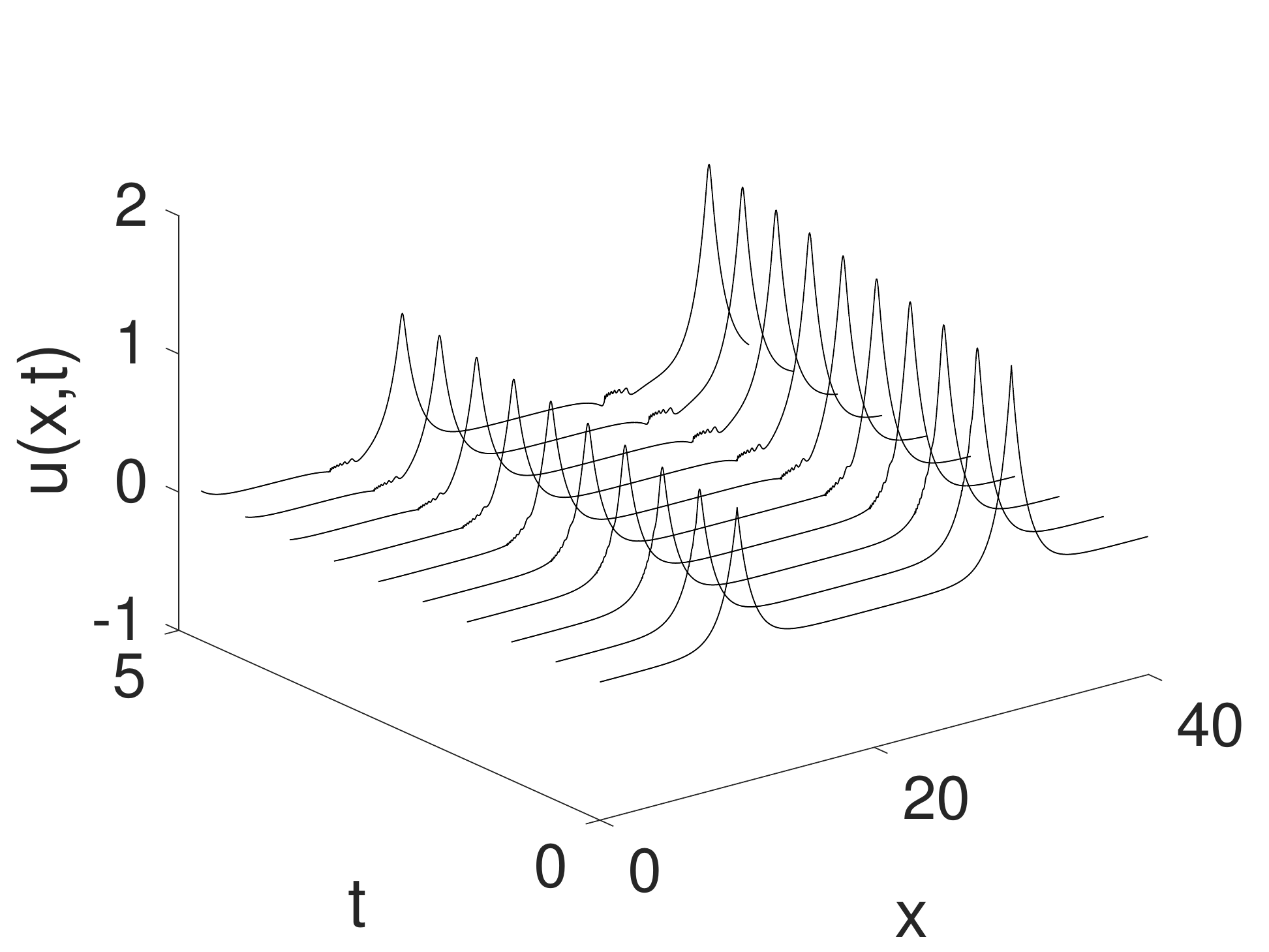}
        \end{subfigure}
 \begin{subfigure}[b]{0.33\textwidth}
      \centering
                \includegraphics[width=0.99\textwidth]{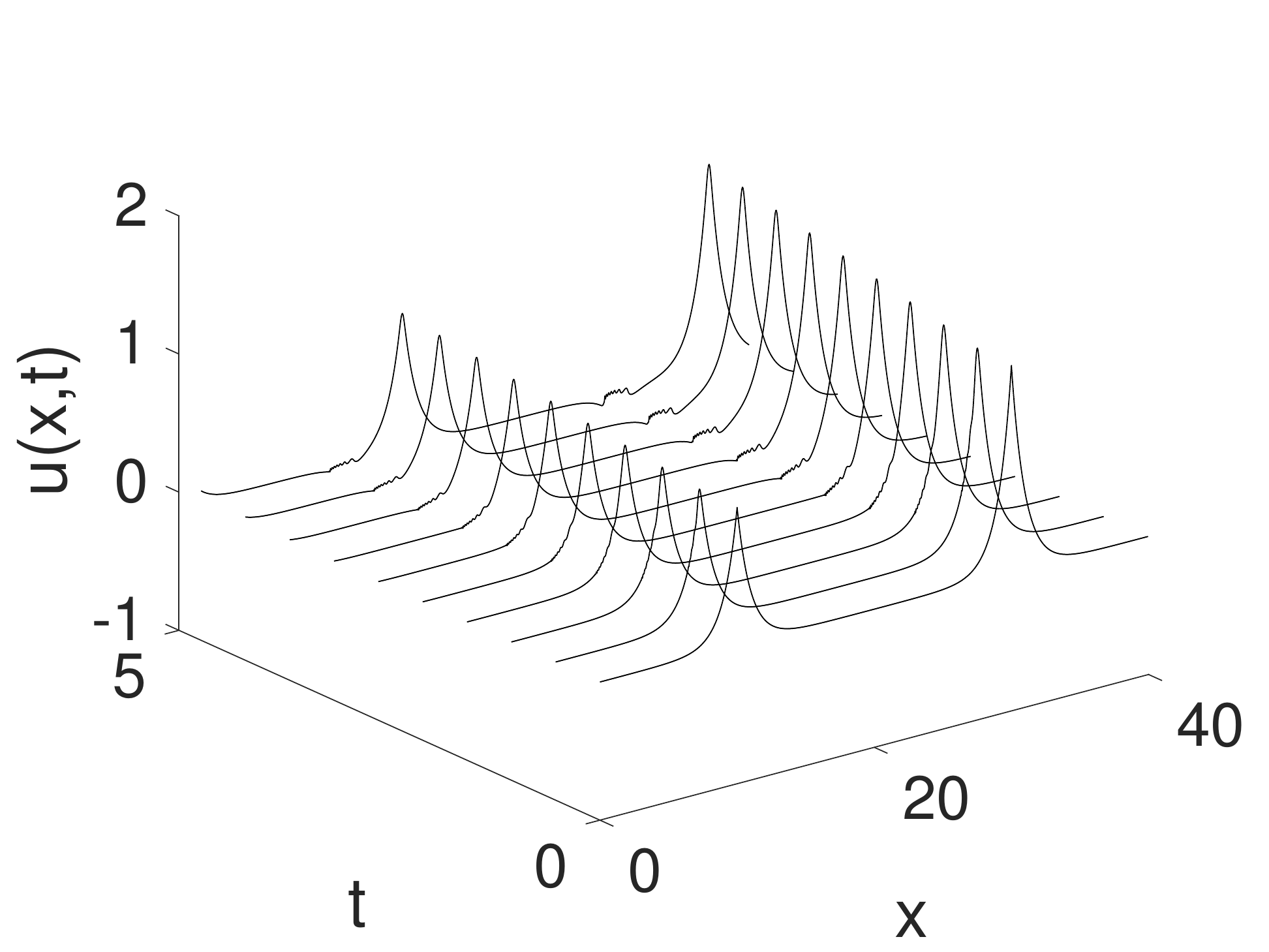}
        \end{subfigure}
      \caption{In this experiment, $\Delta x=0.04$, $\Delta t=0.0002$. \textbf{Left:} relative energy errors,  \textbf{middle:} propagation of the wave by PDGM, \textbf{right:} propagation of the wave by Kahan's method.}\label{ener_solution-CH2peakon}
\end{figure}

\begin{figure}[H]
\hspace{-10pt}
\centering
\vspace{10pt}
      \begin{subfigure}[b]{0.33\textwidth}
      \centering
                \includegraphics[width=0.99\textwidth]{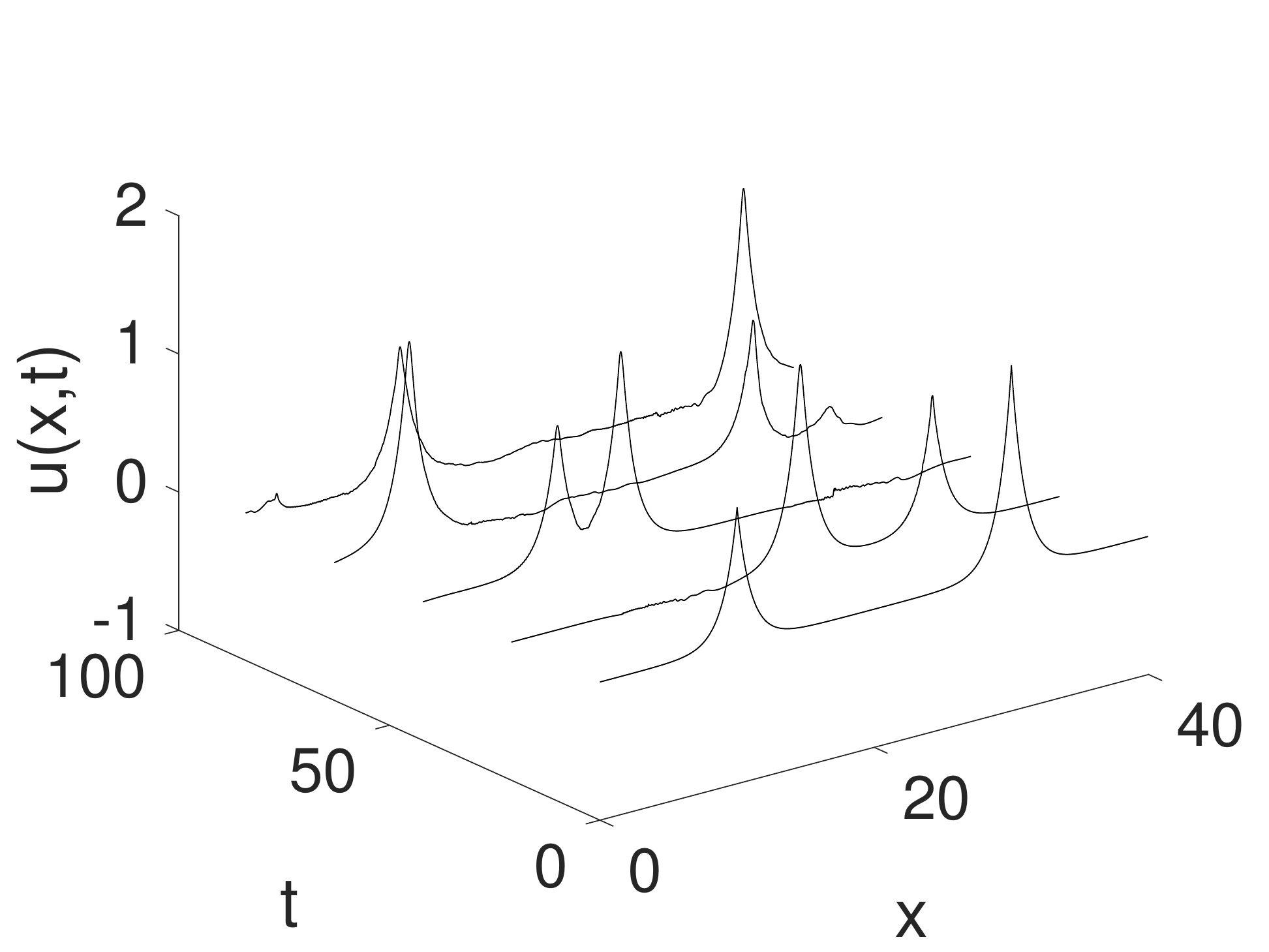}
        \end{subfigure}
        \begin{subfigure}[b]{0.33\textwidth}
        \centering
                \includegraphics[width=0.99\textwidth]{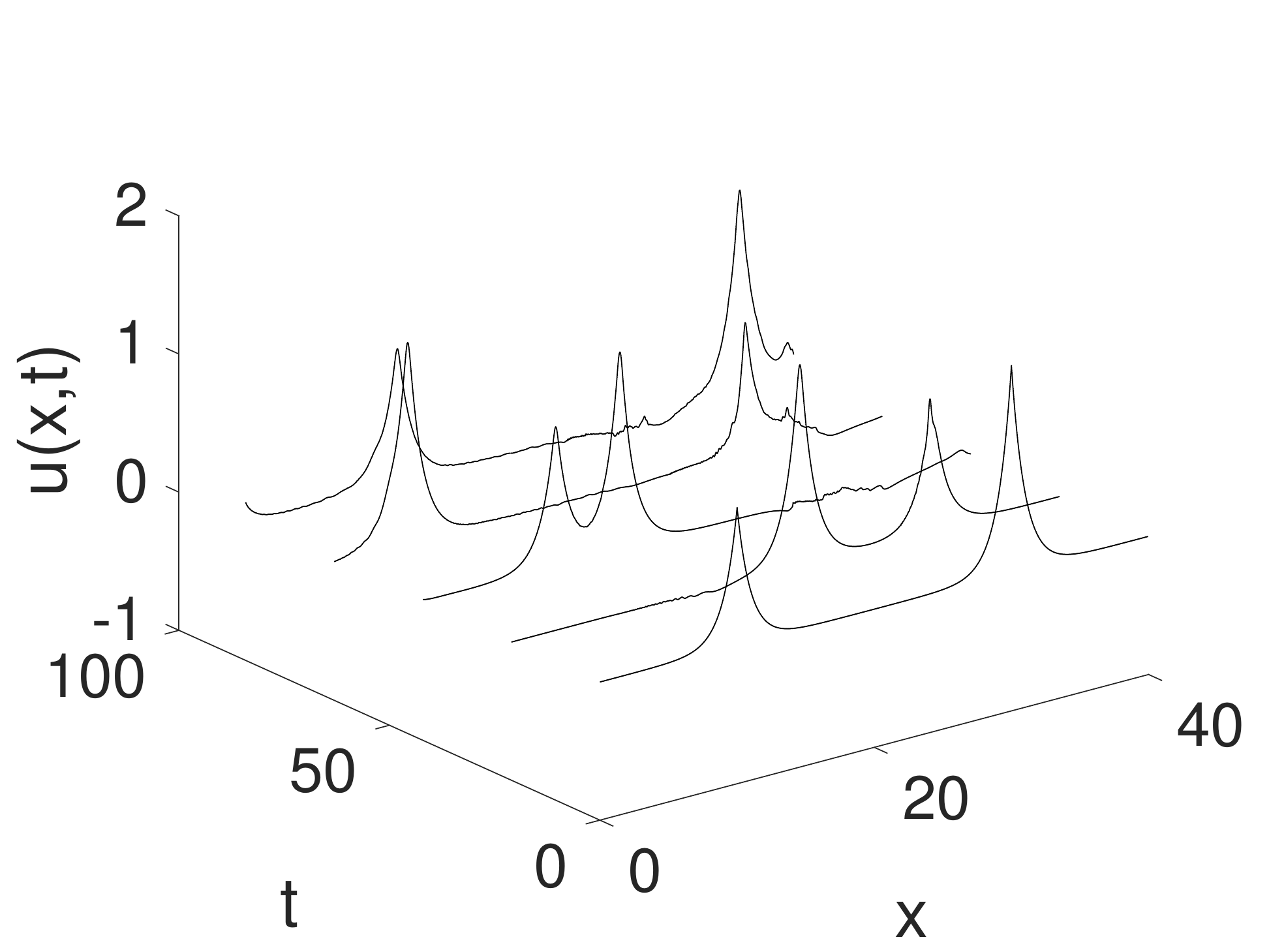}
        \end{subfigure}
 \begin{subfigure}[b]{0.33\textwidth}
      \centering
                \includegraphics[width=0.99\textwidth]{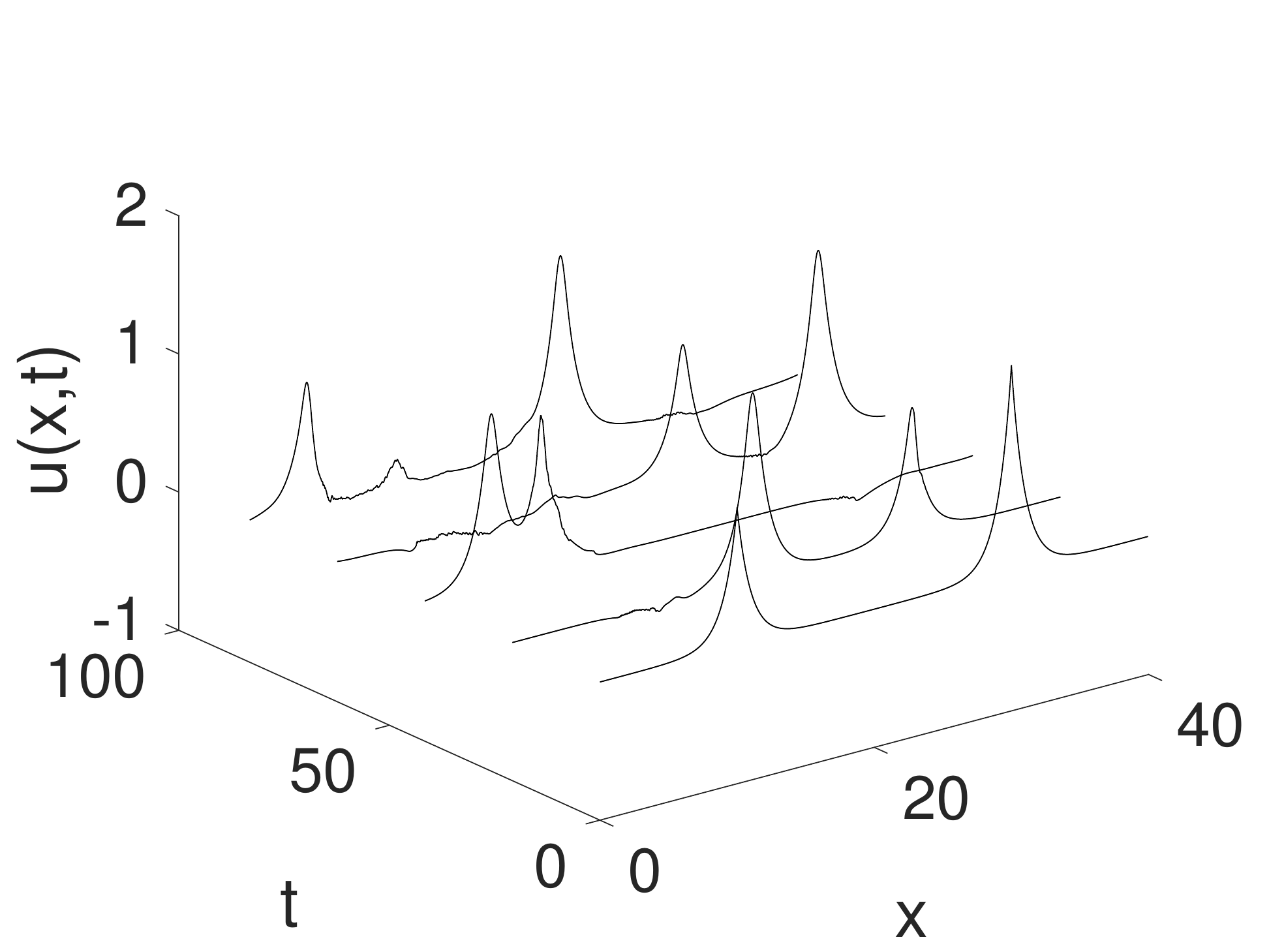}
        \end{subfigure}
      \caption{In this experiment, $\Delta x=0.04$. \textbf{Left:} propagation of the wave by PDGM, $\Delta t=0.02$, \textbf{middle:} propagation of the wave by Kahan's method, $\Delta t=0.02$, \textbf{right:} propagation of the wave by Kahan's method, $\Delta t=0.2$.}\label{solution-CH-2peakon-Tkahan}
\vspace{10pt}
\end{figure}

\subsection {Korteweg--de Vries equation} 

For the Camassa--Holm equation, the vector field of the semi-discretized system is a homogeneous quadratic polynomial. In this section, we deal with the KdV equation, for which the vector field of the semi-discretized equation is a non-homogeneous quadratic polynomial. Kahan's method has also previously been used for the temporal discretization of this equation, see \cite{kahan1997unconventional}.

The KdV equation
\begin{equation}\label{eq_kdv}
u_t + 6 u u_x + u_{xxx} = 0
\end{equation}
on the periodic domain $ \Sb := \RR / L \mathbb{Z} $ has the conserved Hamiltonians
\begin{align*}
\mathcal{H}_1 (u(t)) = \frac{1}{2} \int_{\Sb} u^2 \, \rd x, \qquad
\mathcal{H}_2 (u(t)) = \int_{\Sb} \left(- u^3 + \frac{1}{2} u_x^2 \right) \rd x.
\end{align*}
In the following we consider the variational form based on the Hamiltonian $\mathcal{H}_2$:
\begin{equation}\label{eq_kdv_E}
 u_t = \partial_x \frac{\delta \mathcal{H}_2}{\delta u }, \qquad \frac{\delta \mathcal{H}_2}{\delta u } = - 3 u^2 - u_{xx}.
\end{equation}

\subsubsection {Numerical schemes for the KdV equation}
We discretize the energy $\mathcal{H}_2$ for the KdV equation \eqref{eq_kdv_E} as
\begin{equation*}
H_2(U)\Delta x=\sum_{k=1}^K\left(-u_k^3+\frac{(\fd_x u_k)^2 + (\bd_x u_k)^2}{4}\right)\Delta x.
\end{equation*}
From simple calculations, the corresponding gradient is given by
\[  \nabla H_2 (U) = \left(-3 U_{\boldsymbol{\cdot}}^2 - \Dc[2]U \right), \] 
and thus we have the semi-discretized form for \eqref{eq_kdv_E}:
\begin{equation}\label{eq_kdv_E_dis}
 \dot{U} = \Dc\left( -3 U_{\boldsymbol{\cdot}}^2 -\Dc[2]U \right).
\end{equation}

Applying the schemes under consideration to \eqref{eq_kdv_E_dis} gives 
\begin{align}
\frac{U^{n+1}-U^n}{\Delta t}=&\Dc\nabla H_2(\frac{U^{n}+U^{n+1}}{2}), \qquad\text{(MP)}\label{semi-ode-KdV_Midpoint}\\
\begin{split}\label{semi-ode-KDV_Kahan_two}
\frac{U^{n+1}-U^n}{\Delta t}=&-\frac{1}{2}\Dc (\nabla H(U^n) + \nabla H(U^{n+1})) \\
&+ 2\Dc\nabla H(\frac{U^n+U^{n+1}}{2}),
\end{split}\qquad\text{(Kahan)}\\
\frac{U^{n+2}-U^n}{2\Delta t}=&\Dc\overline{\nabla} \tilde{H}_2 (U^{n},U^{n+1},U^{n+2}),\qquad\text{(PDGM)}\label{discreteE-KdV_polar}
\end{align}
where $H_2^{''}(U)=-6 \,\text{diag}(U)-\Dc[2]$ is the Hessian of $H_2(U)$ and $\overline{\nabla} \tilde{H}_2 (U^{n},U^{n+1},U^{n+2})$ is found as in Proposition \ref{pdg}, with polarised discrete energy 
\begin{align*}
\tilde{H}_2 (u_k^n,u_k^{n+1}) \Delta x := &\sum_{k=1}^K ( -u_k^{n} u_k^{n+1} \frac{u_k^{n}  + u_k^{n+1} }{2} + \frac{a}{2} ( \fd_x u_k^{n}) (\fd_x u_k^{n+1}) \\
&+\frac{1-a}{2} \frac{ (\fd_x u_k^{n})^2 + ( \fd_x u_k^{n+1})^2 }{2} ) \Delta x.
\end{align*}
\begin{remark}
We perform several numerical simulations to find a good choice of the parameter $a$, and we take $a=-\frac{1}{2}$ for PDGM in the following numerical examples for the KdV equation.
\end{remark}

\subsubsection {Stability analysis of the schemes}
To analyse the stability of the above methods, we perform the von Neumann stability analysis for the Kahan and PDGM schemes applied to the linearized form of the KdV equation \eqref{eq_kdv}
\begin{align}\label{eq_kdv_linearized eq}
u_t+ u_{xxx} = 0.
\end{align}

The equation for the amplification factor for Kahan's method is
\begin{align*}
(1+i\lambda(\cos\theta-1)\sin\theta)g+i\lambda (\cos\theta-1)\sin\theta-1=0,
\end{align*}
and its root is
\begin{align*}
g=\frac{1-i\lambda (\cos\theta-1)\sin\theta}{1+i\lambda (\cos\theta-1)\sin\theta},
\end{align*}
where $\lambda:=\frac{\Delta t}{{\Delta x}^3}$.
Since $g$ is a simple root on the unit circle, Kahan's method is unconditionally stable for the linearized KdV equation.

The equation for the amplification
factor for PDGM is
\begin{align}\label{LMI_stabi_linearized KdV}
g^2-1+i\lambda(3g^2-2g+3)(\cos\theta-1)\sin\theta=0.
\end{align}
The two roots of the above equation are thus
\begin{align*}
g_{1}&=\frac{3b^2+\sqrt{1+8b^2}+ib(3\sqrt{1+8b^2}-1)}{1+9b^2},\\
g_{2}&=\frac{3b^2-\sqrt{1+8b^2}-ib(3\sqrt{1+8b^2}+1)}{1+9b^2},
\end{align*}
where $b=\lambda (1-\cos\theta)\sin\theta$. We observe that $\lvert g_1\rvert=\lvert g_2\rvert=1$, and $g_1 \neq g_2$, therefore PDGM is unconditionally stable for the linearized KdV equation.

\subsubsection {Numerical tests for the KdV equation}\label{Numericla example kdv}

\noindent \textbf{Example 1 (One soliton solution):} Consider the initial value
$$u(x,0)=2 \, \text{sech}^2 (x-L/2),$$
where $x\in[0,L]$, $L=40$. We apply our schemes over the time interval $[0,T]$, $T=100$, with step sizes $\Delta x=0.05$, $\Delta t=0.0125$. From our observations, all the methods behave well. The shape of the wave is well kept by all the methods, also for long time integration, see Figure \ref{shape-phase-global error-kdv-1soliton}. 
The energy errors of all the methods are rather small and do not increase over long time integration, see Figure \ref{energy error-solution-kdv-1soliton} (left). 
We then use a coarser time grid, $\Delta t=0.035$, and both methods are still stable, see Figure \ref{solution-kdv-1soliton-TKahan-LMI} (left two). However we observe that the global error of PDGM becomes much bigger than that of Kahan's method. When an even larger time step-size, $\Delta t=0.04$, is considered, the solution for PDGM blows up while the solution for Kahan's method is rather stable. In this case, the PDG method applied to the nonlinear KdV equation is unstable and the numerical solution blows up at around $t=8$. Even if we increase the time step-size to $\Delta t=0.1$, Kahan's method still works well, see Figure \ref{solution-kdv-1soliton-TKahan-LMI} (middle). When $\Delta t=0.15$ is considered, we observe evident signs of instability in the solution of Kahan's method. The solution will blow up rapidly when $\Delta t=0.2\gg\Delta x$.

\begin{figure}[h]
\vspace{-10pt}
\hspace{-10pt}
\centering
      \begin{subfigure}[b]{0.33\textwidth}
      \centering
                \includegraphics[width=0.99\textwidth]{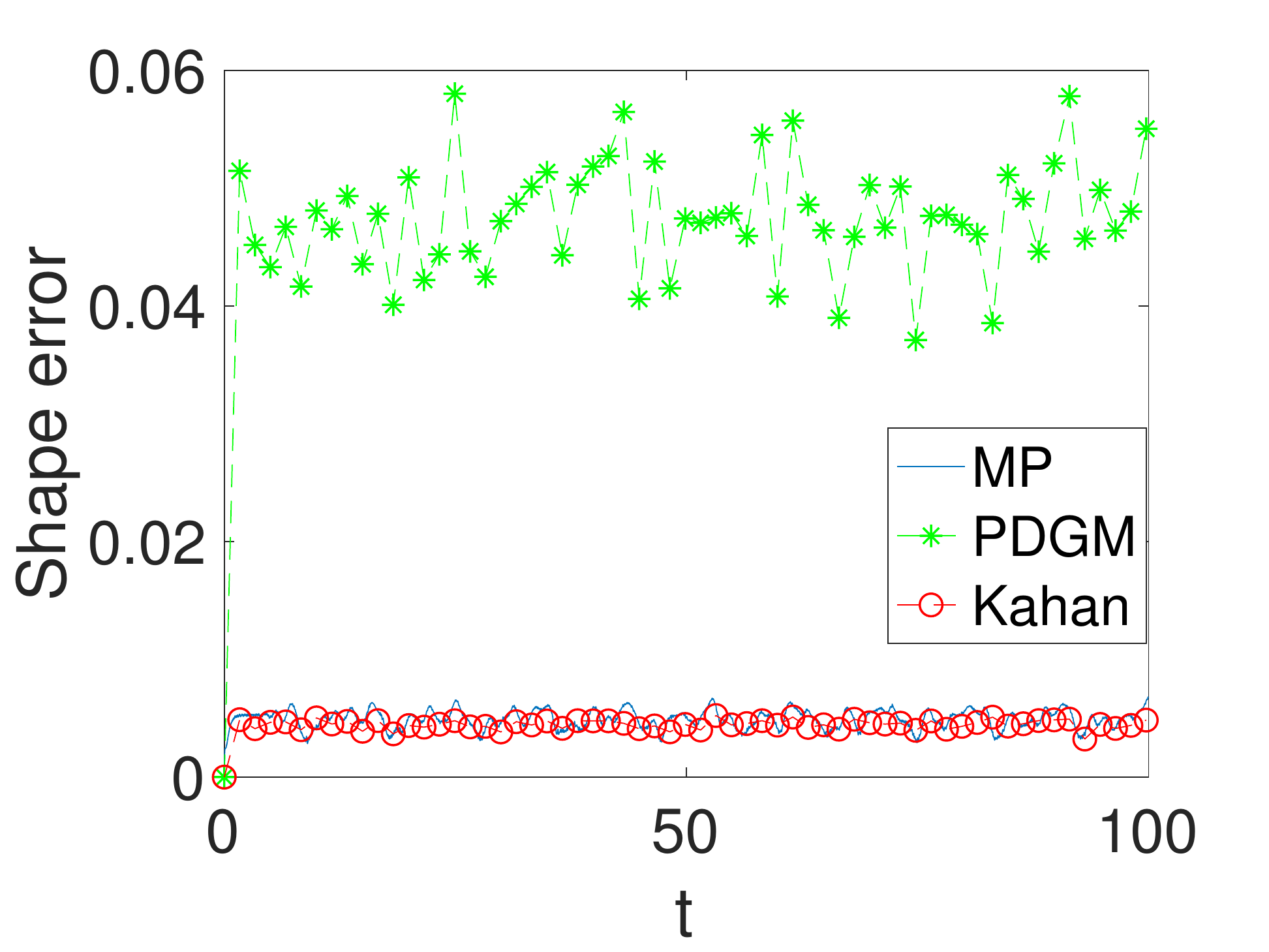}
        \end{subfigure}
        \begin{subfigure}[b]{0.33\textwidth}
        \centering
                \includegraphics[width=0.99\textwidth]{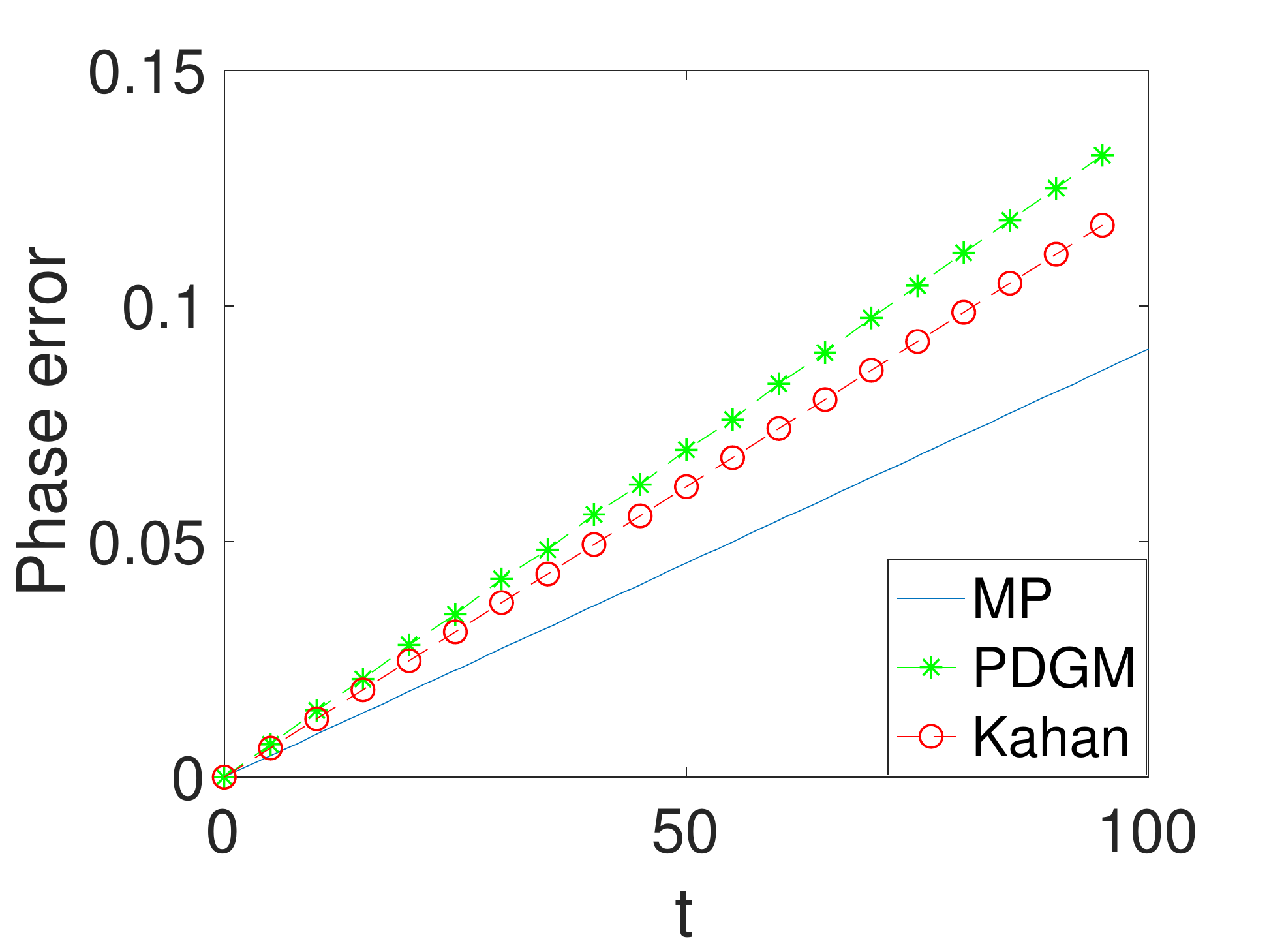}
        \end{subfigure}
 \begin{subfigure}[b]{0.33\textwidth}
      \centering
                \includegraphics[width=0.99\textwidth]{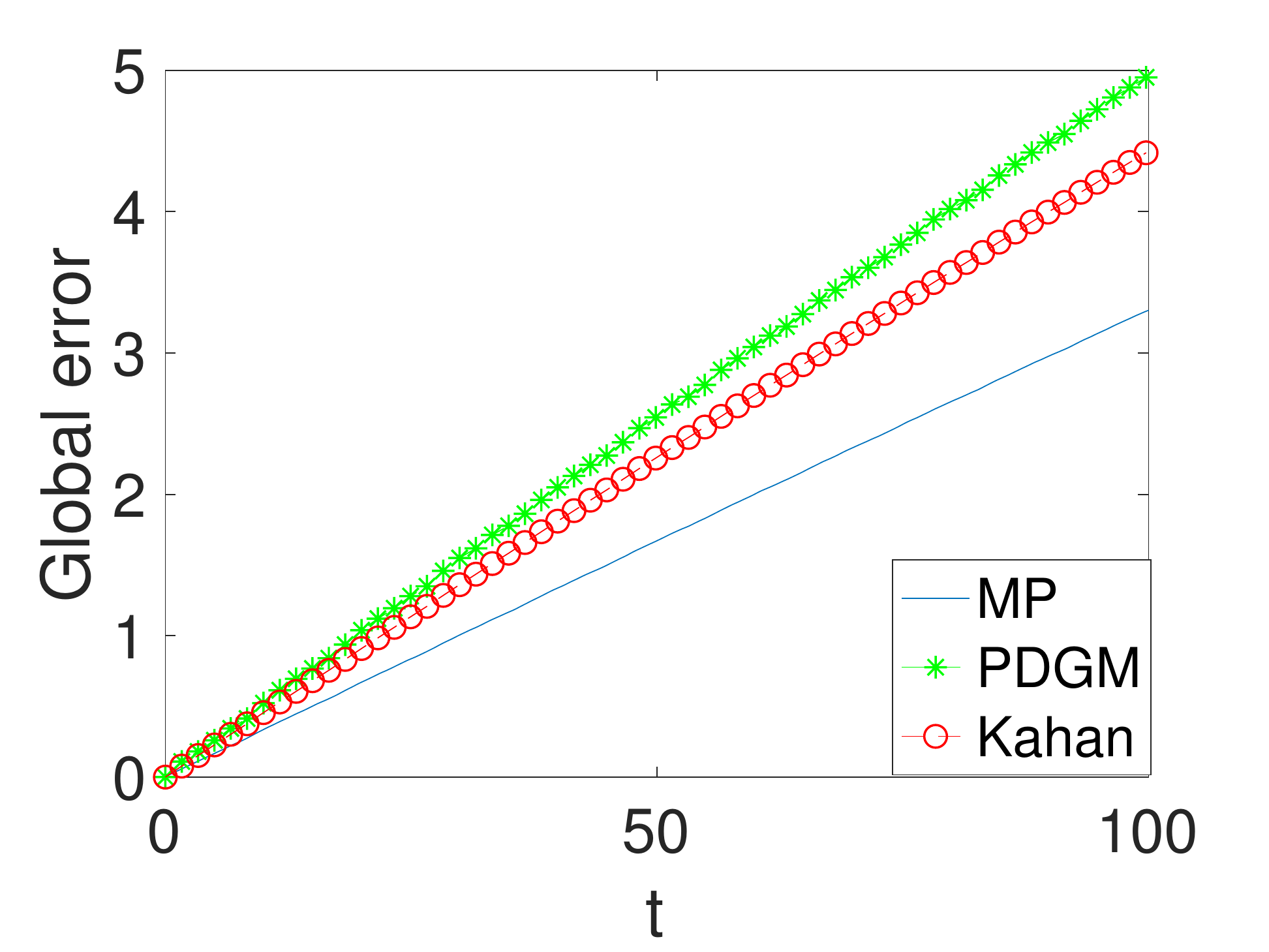}
        \end{subfigure}
      \caption{Space step size $\Delta x=0.05$, time step size $\Delta t=0.0125$. \textbf{Left:} shape error, \textbf{middle:} phase error, \textbf{right:} global error.}\label{shape-phase-global error-kdv-1soliton}
\vspace{10pt}
\end{figure}
\begin{figure}[h]
\hspace{-10pt}
\centering
\vspace{10pt}
      \begin{subfigure}[b]{0.33\textwidth}
      \centering
                \includegraphics[width=0.99\textwidth]{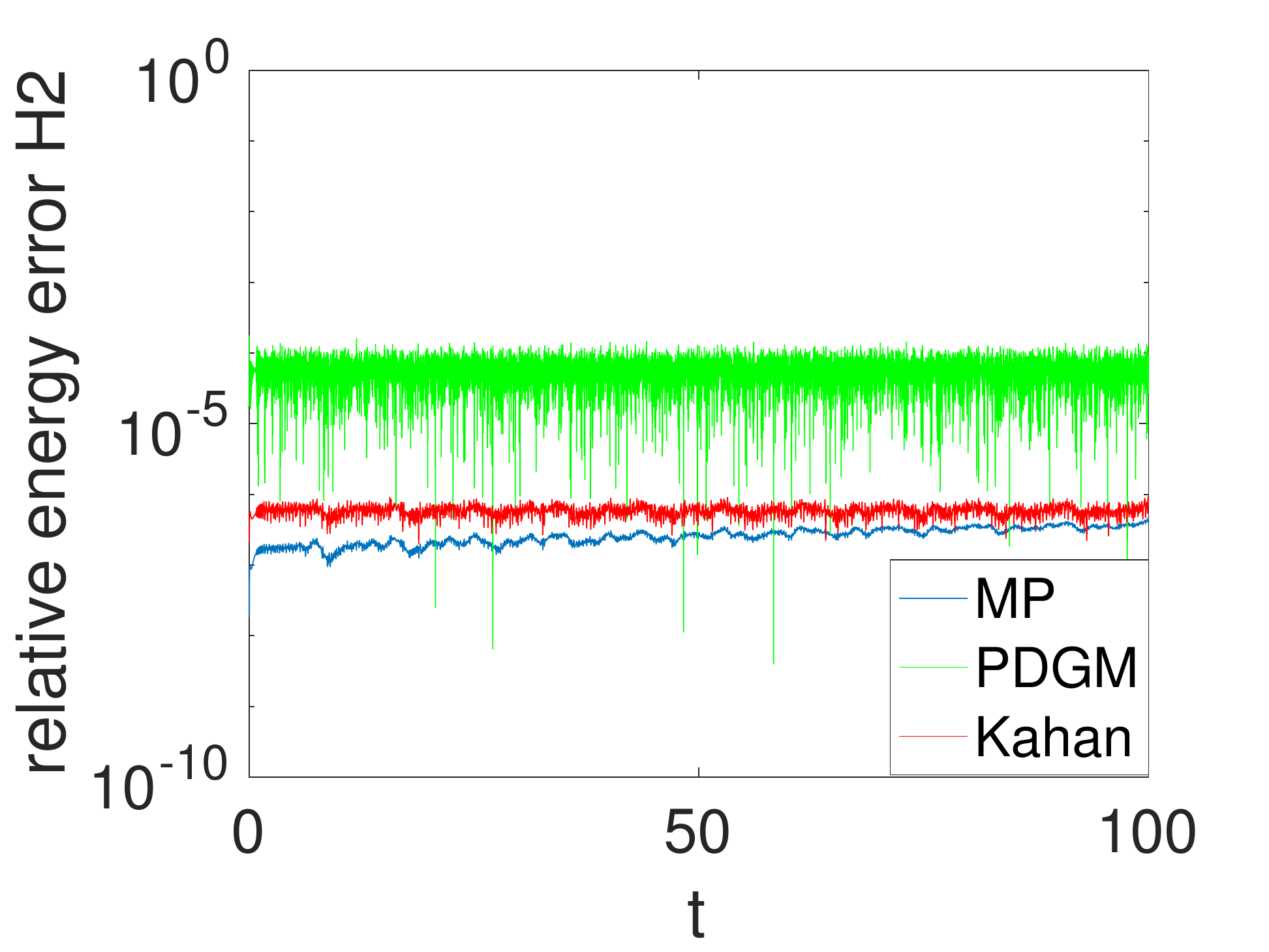}
        \end{subfigure}
        \begin{subfigure}[b]{0.33\textwidth}
        \centering
                \includegraphics[width=0.99\textwidth]{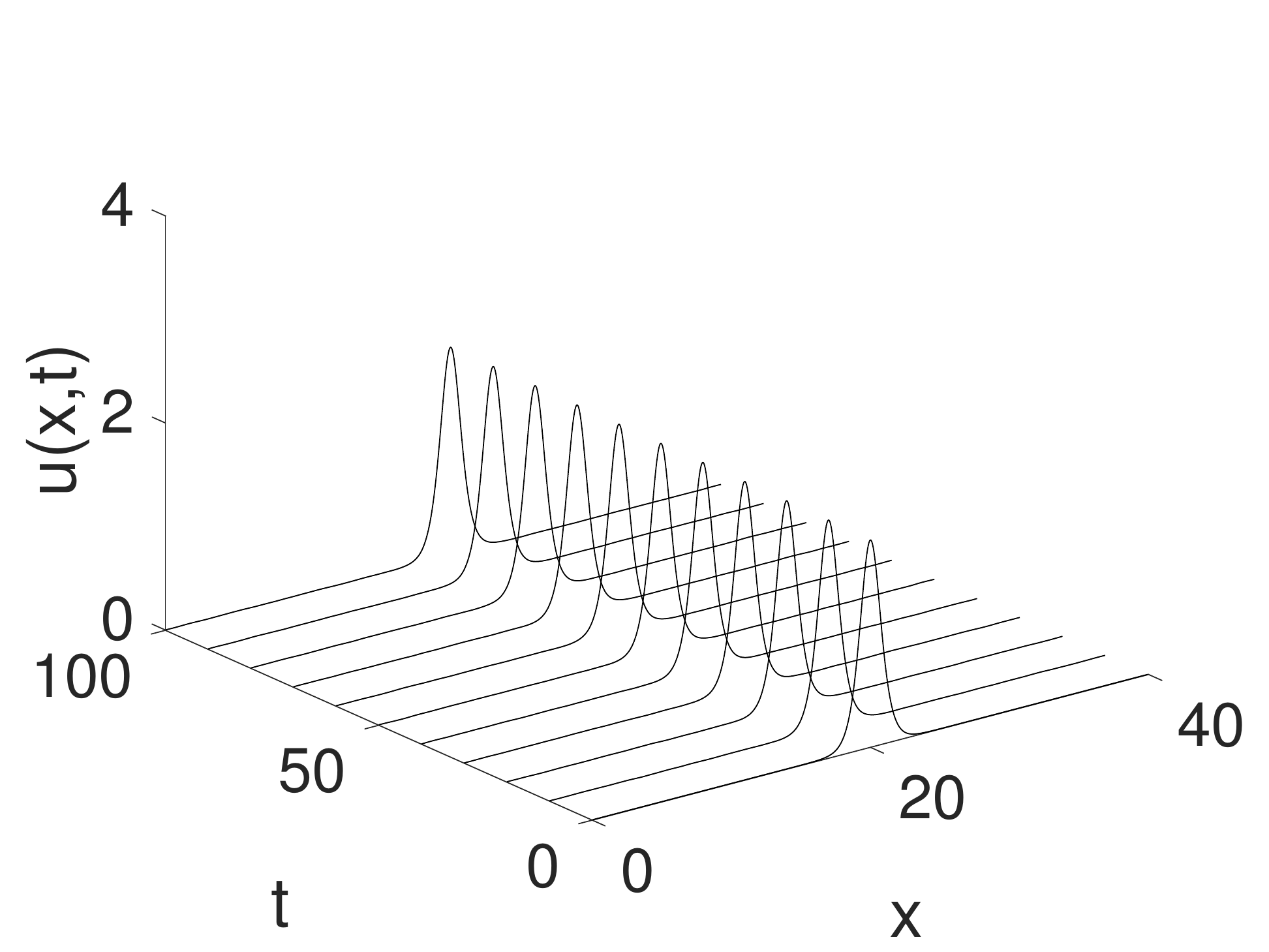}
        \end{subfigure}
 \begin{subfigure}[b]{0.33\textwidth}
      \centering
                \includegraphics[width=0.99\textwidth]{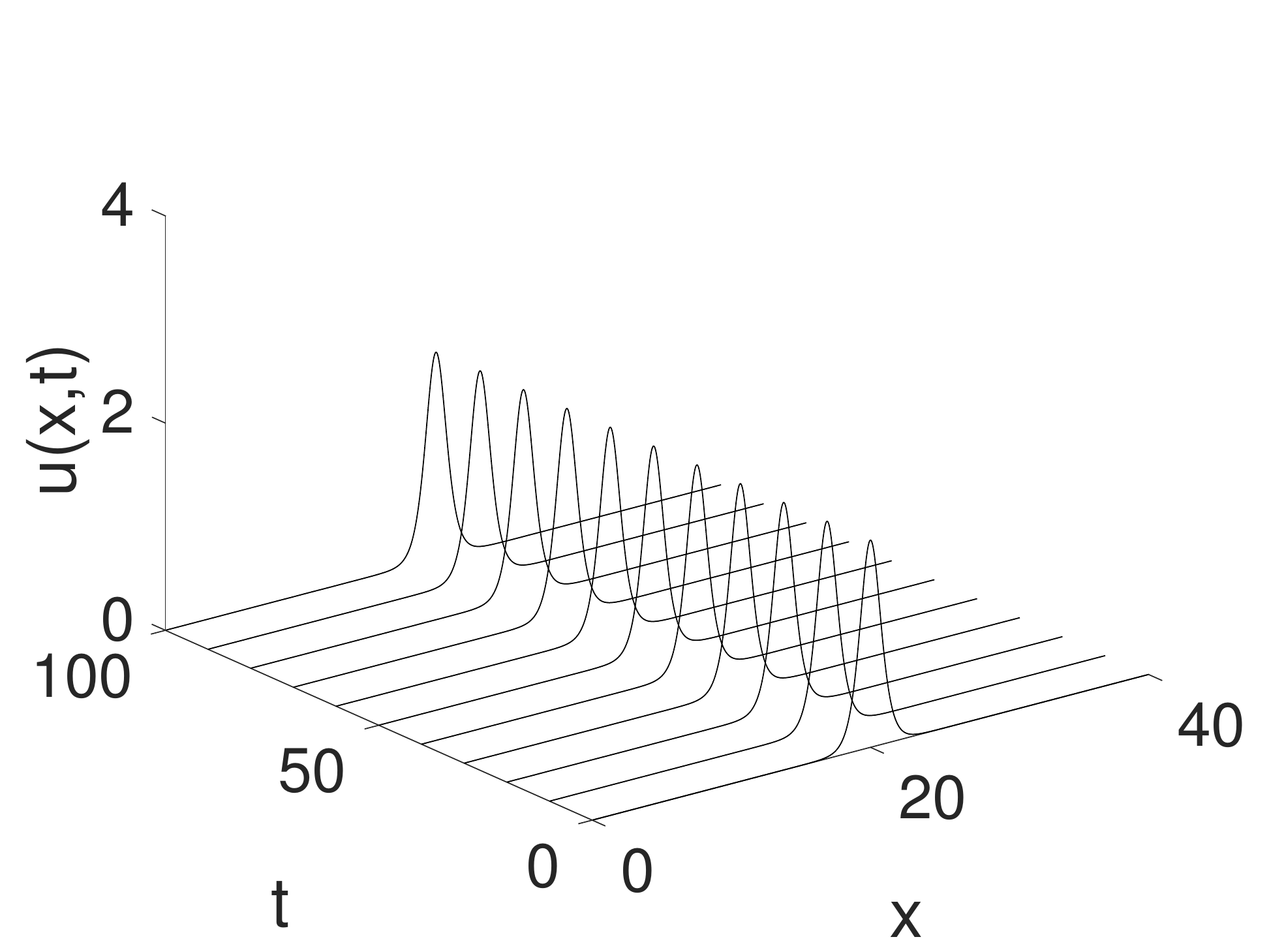}
        \end{subfigure}
      \caption{With $\Delta x=0.05$, $\Delta t=0.0125$. \textbf{Left:} relative energy errors, \textbf{right two:} propagation of the wave by PDGM and Kahan's method.}\label{energy error-solution-kdv-1soliton}
\vspace{10pt}
\end{figure}

\begin{figure}[h]
\vspace{-10pt}
\hspace{-10pt}
\centering
      \begin{subfigure}[b]{0.33\textwidth}
      \centering
                \includegraphics[width=0.99\textwidth]{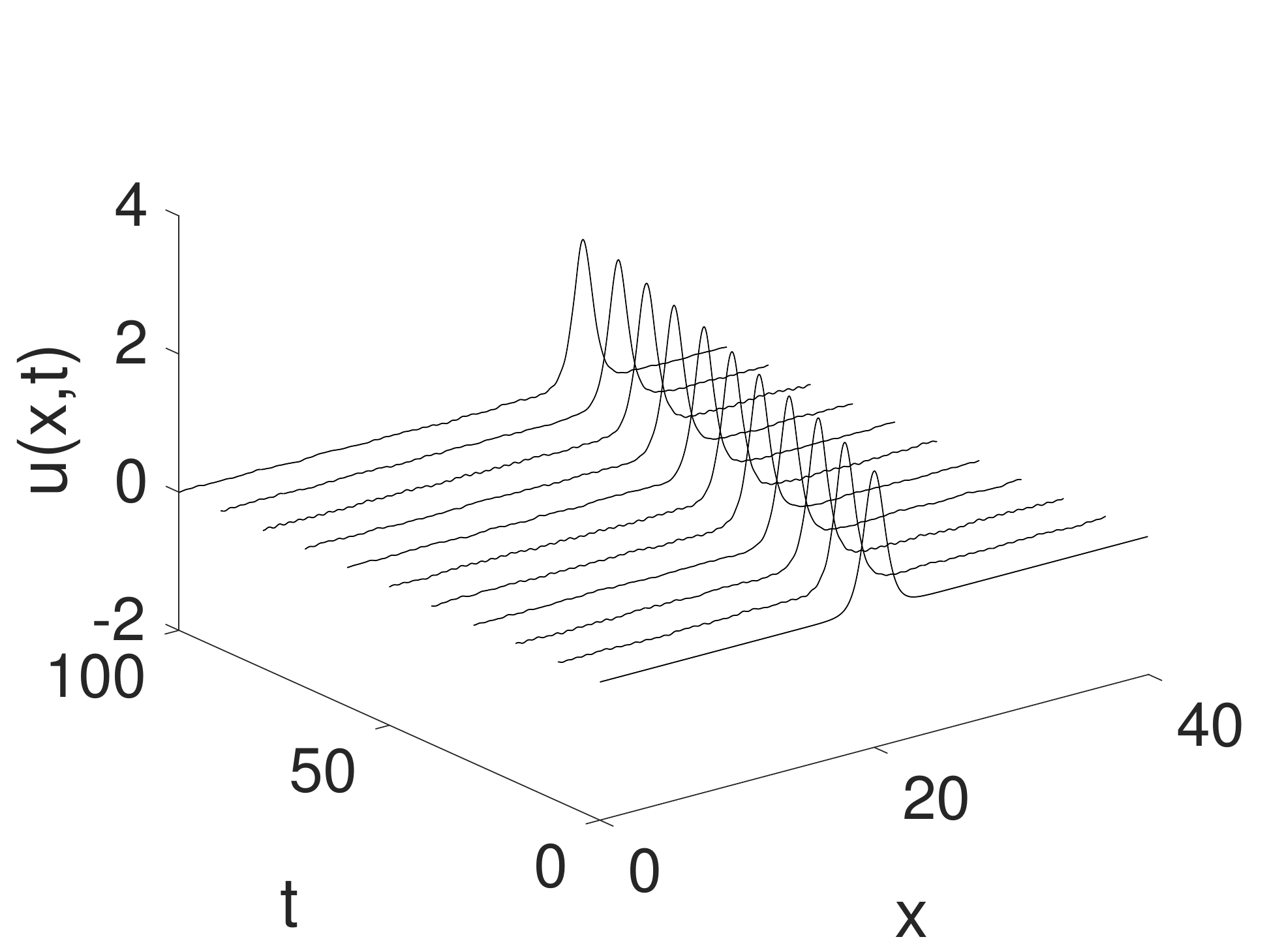}
        \end{subfigure}
        \begin{subfigure}[b]{0.33\textwidth}
        \centering
                \includegraphics[width=0.99\textwidth]{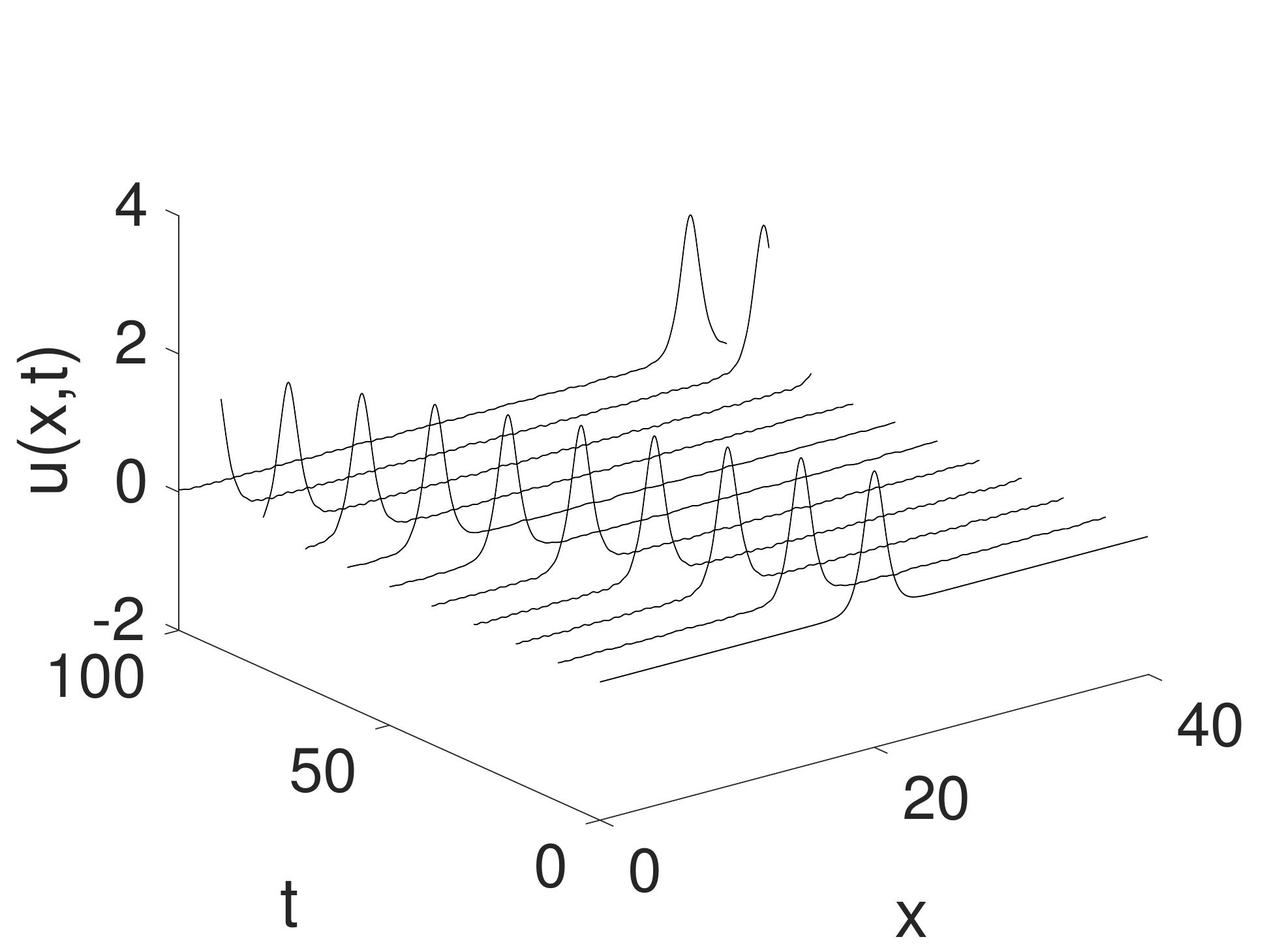}
        \end{subfigure}
        \begin{subfigure}[b]{0.33\textwidth}
      \centering
                \includegraphics[width=0.99\textwidth]{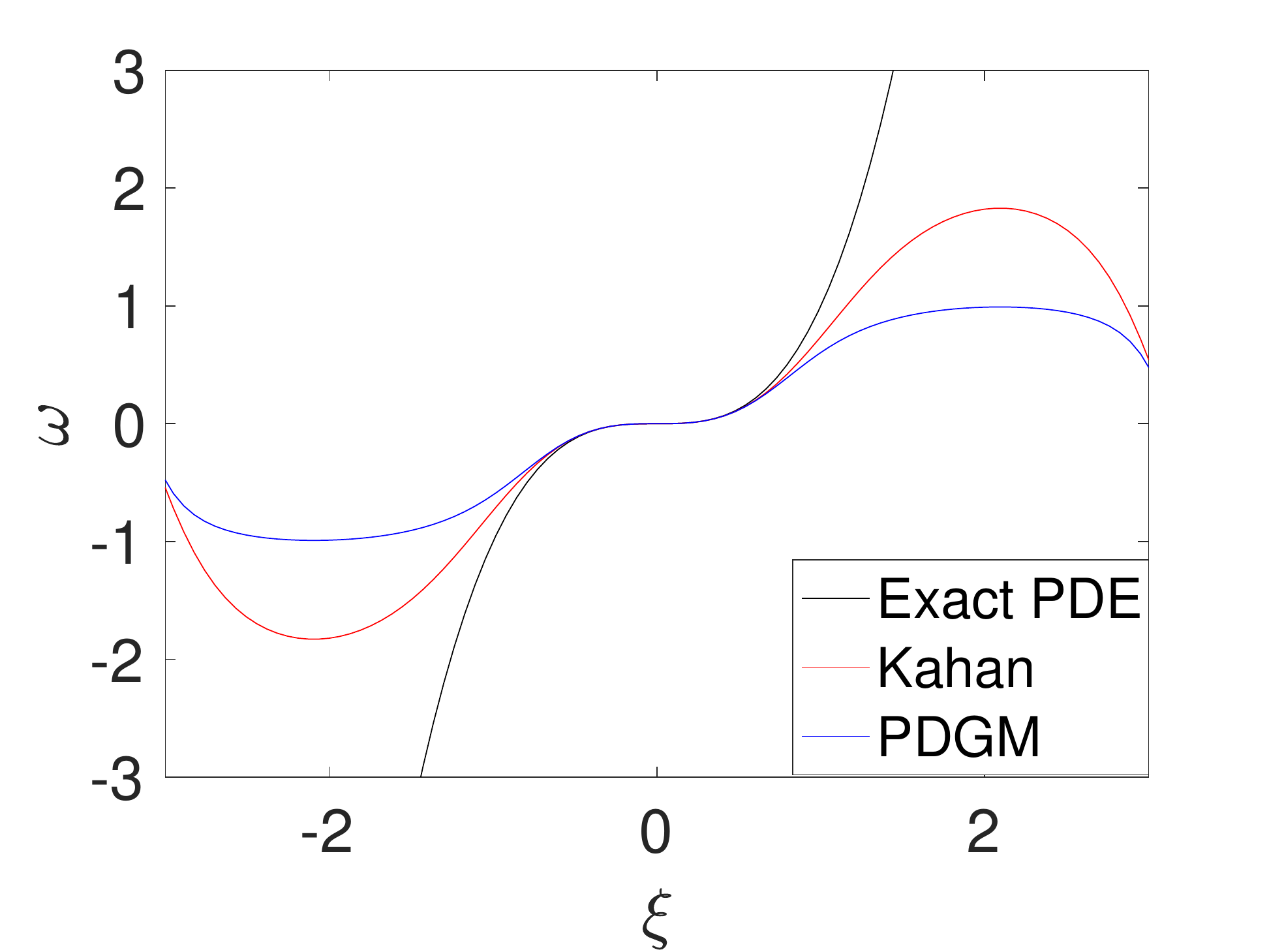}
        \end{subfigure}
      \caption{With $\Delta x=0.05$. \textbf{Left:} $\Delta t=0.035$, propagation of the wave by PDGM, \textbf{middle:} $\Delta t=0.1$, propagation of the wave by Kahan's method, \textbf{right:} dispersion relation for $\lambda=1$.}\label{solution-kdv-1soliton-TKahan-LMI}
\vspace{10pt}
\end{figure}

\vspace{8pt}
\noindent \textbf{Example 2 (Two solitons solution):} We choose initial value
$$u(x,0)=6\,\text{sech}^2x,$$
and consider periodic boundary conditions $u(0,t)=u(L,t)$, where $x\in[0,L]$, $L=40$. We set the space step size $\Delta x=0.05$ and apply the aforementioned schemes on time interval $[0,T]$ with $T=100$, $\Delta t=0.001$. All the methods behave stably. The profiles of Kahan's method and the midpoint method are almost indistinguishable, and the profiles for the midpoint method are thus not presented here. Kahan's method and PDGM preserve the modified energy, and accordingly the energy error of all the methods are rather small over long time integration, see Figure \ref{energy-solution-kdv-2soliton} (left). After a short while the solution has two solitons; one is tall and the other is shorter, see Figure \ref{energy-solution-kdv-2soliton} (the right two plots).  

When we consider a coarser time grid, $\Delta t=0.00375$, both methods are still stable, see Figure \ref{solution-kdv-2soliton-TKahan-LMI} (the left two plots). However, there appear more small wiggles in the solution by PDGM and we observe that the solution of PDGM will blow up rather soon, around $t=1$, for an even coarser time grid $\Delta t=0.005$. When we increase the time step size to $\Delta t=0.0125$ and consider $T=100$, the shape of the exact solution is still well preserved by Kahan's method, even though there appear some small wiggles in the solution at around $t=100$. We observe that the solution of Kahan's method will blow up when $\Delta t=0.05$ is considered. Similar experiments as in this subsection, but for the multisymplectic box schemes, can be found in a paper by Ascher and McLachlan \cite{ascher2005symplectic}. However, here we consider even coarser time grid than there, and the numerical results show that Kahan's method is quite stable, even though it is linearly implicit.

\begin{figure}[H]
\hspace{-10pt}
\centering
\vspace{10pt}
      \begin{subfigure}[b]{0.33\textwidth}
      \centering
                \includegraphics[width=0.99\textwidth]{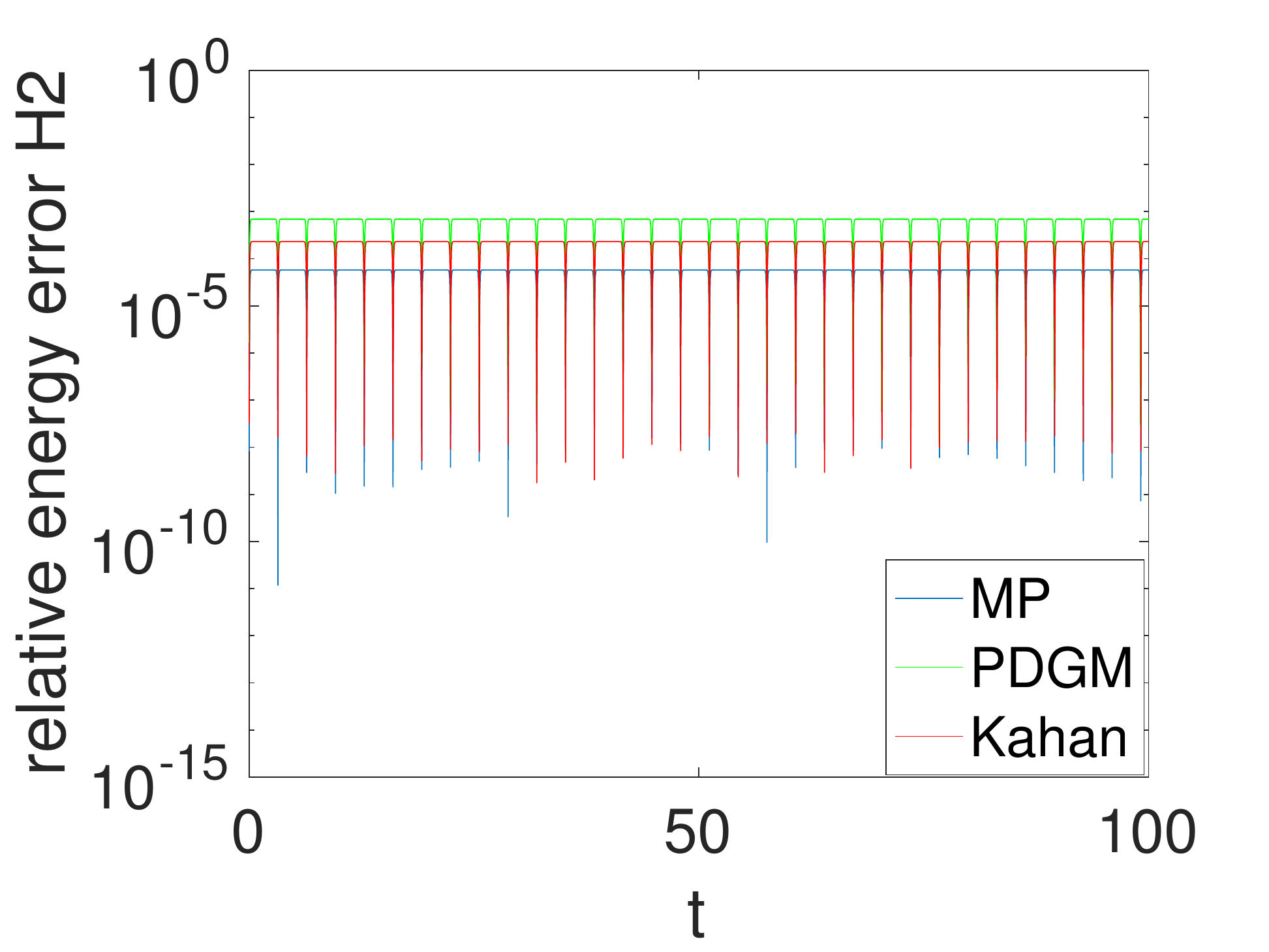}
        \end{subfigure}
        \begin{subfigure}[b]{0.33\textwidth}
        \centering
                \includegraphics[width=0.99\textwidth]{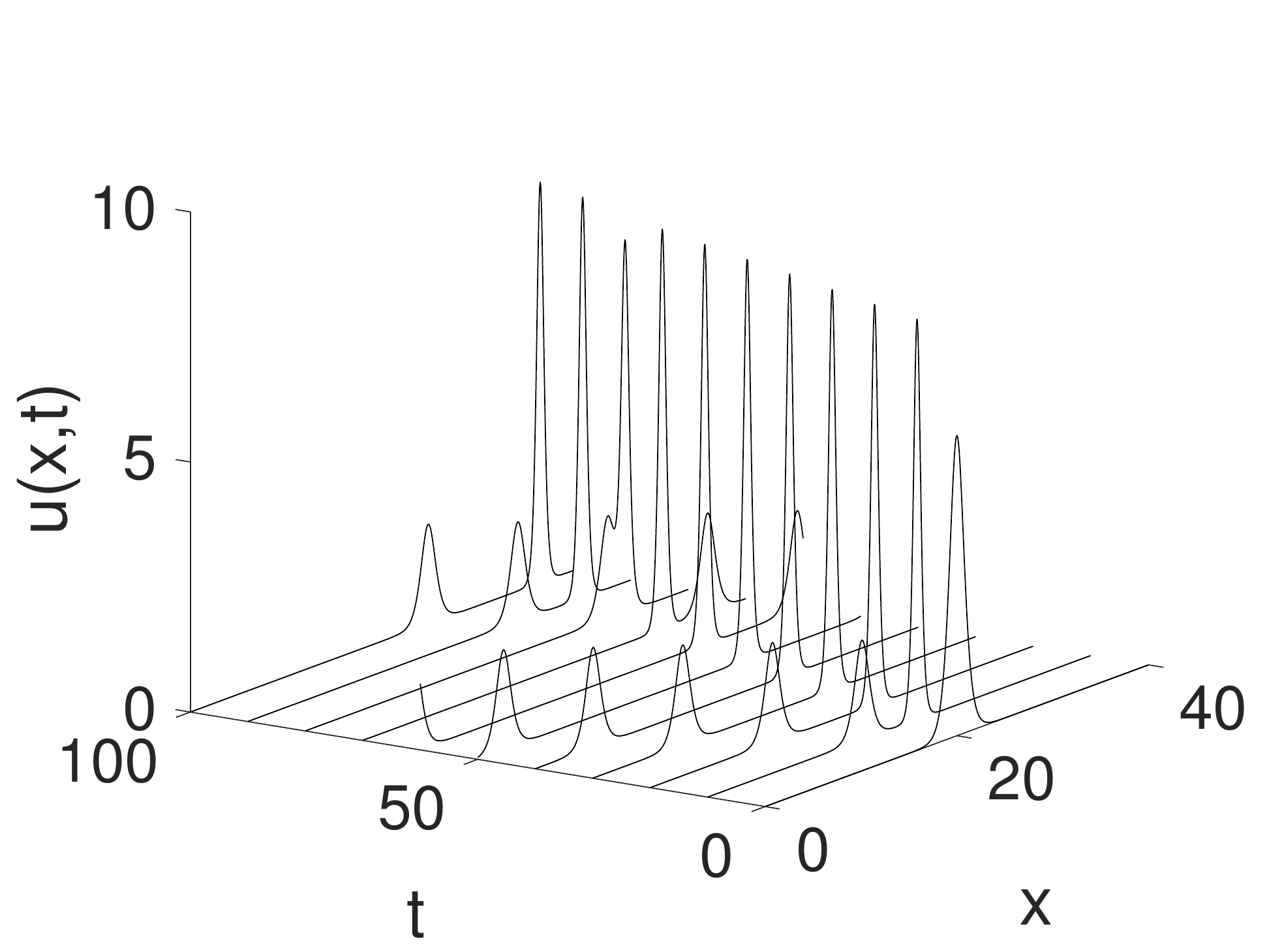}
        \end{subfigure}
 \begin{subfigure}[b]{0.33\textwidth}
      \centering
                \includegraphics[width=0.99\textwidth]{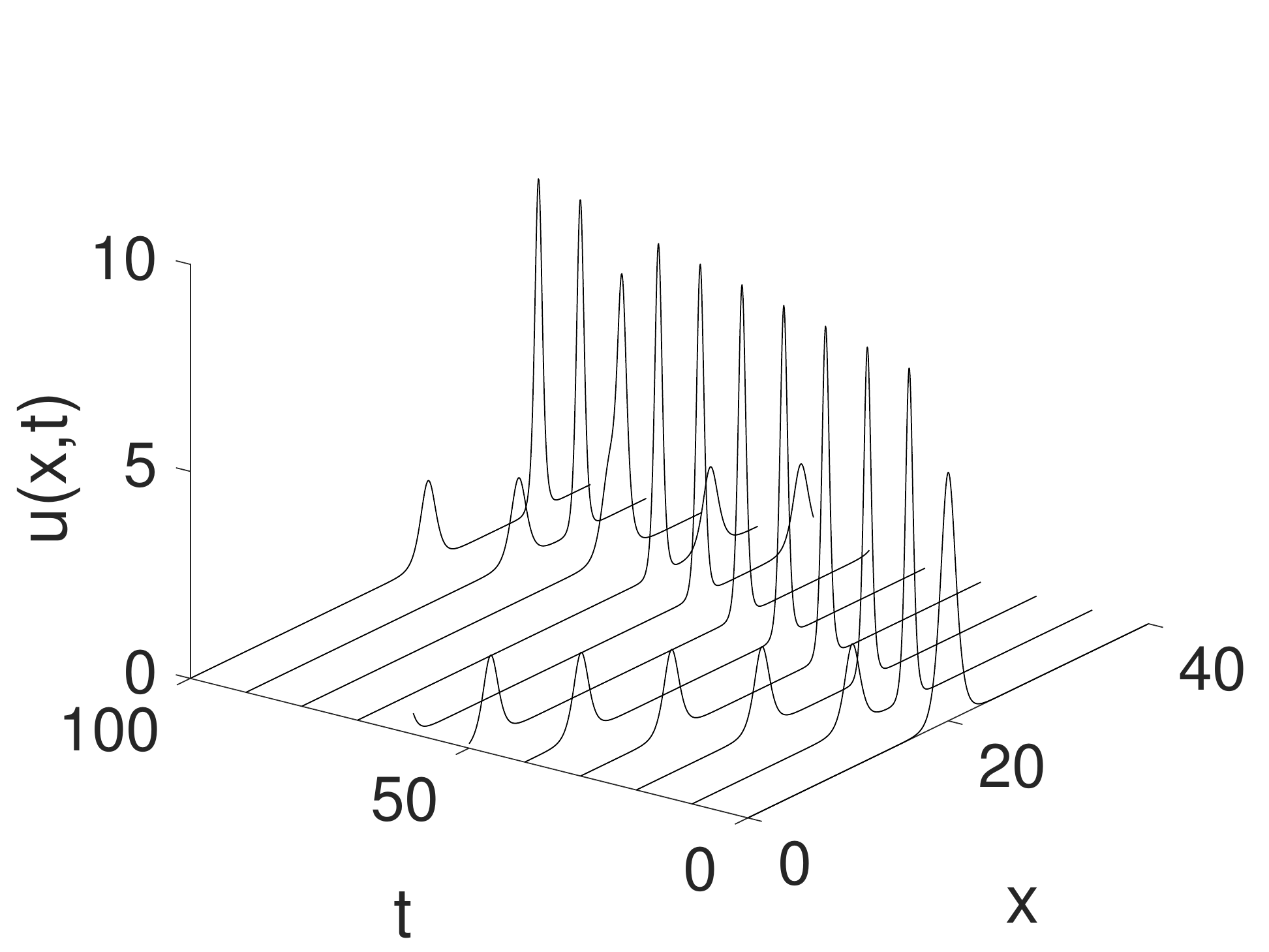}
        \end{subfigure}
      \caption{In this experiment, $\Delta x=0.05$, $\Delta t=0.001$. \textbf{Left:} relative energy errors, \textbf{right two:} propagation of the wave by PDGM and Kahan's method.}\label{energy-solution-kdv-2soliton}
\vspace{10pt}
\end{figure}

\begin{figure}[h!]
\hspace{-10pt}
\centering
\vspace{10pt}
      \begin{subfigure}[b]{0.33\textwidth}
      \centering
                \includegraphics[width=0.99\textwidth]{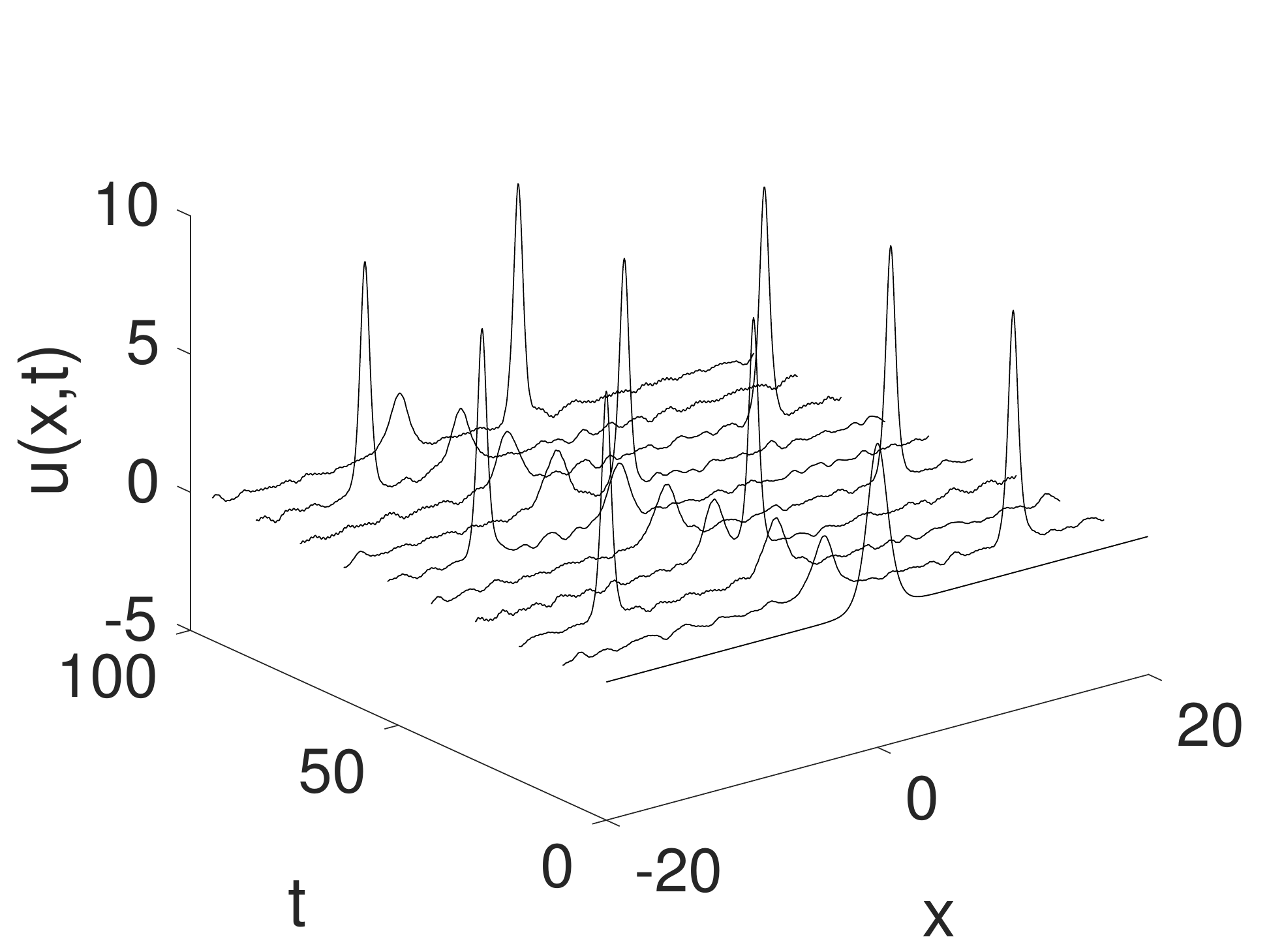}
        \end{subfigure}
        \begin{subfigure}[b]{0.33\textwidth}
        \centering
                \includegraphics[width=0.99\textwidth]{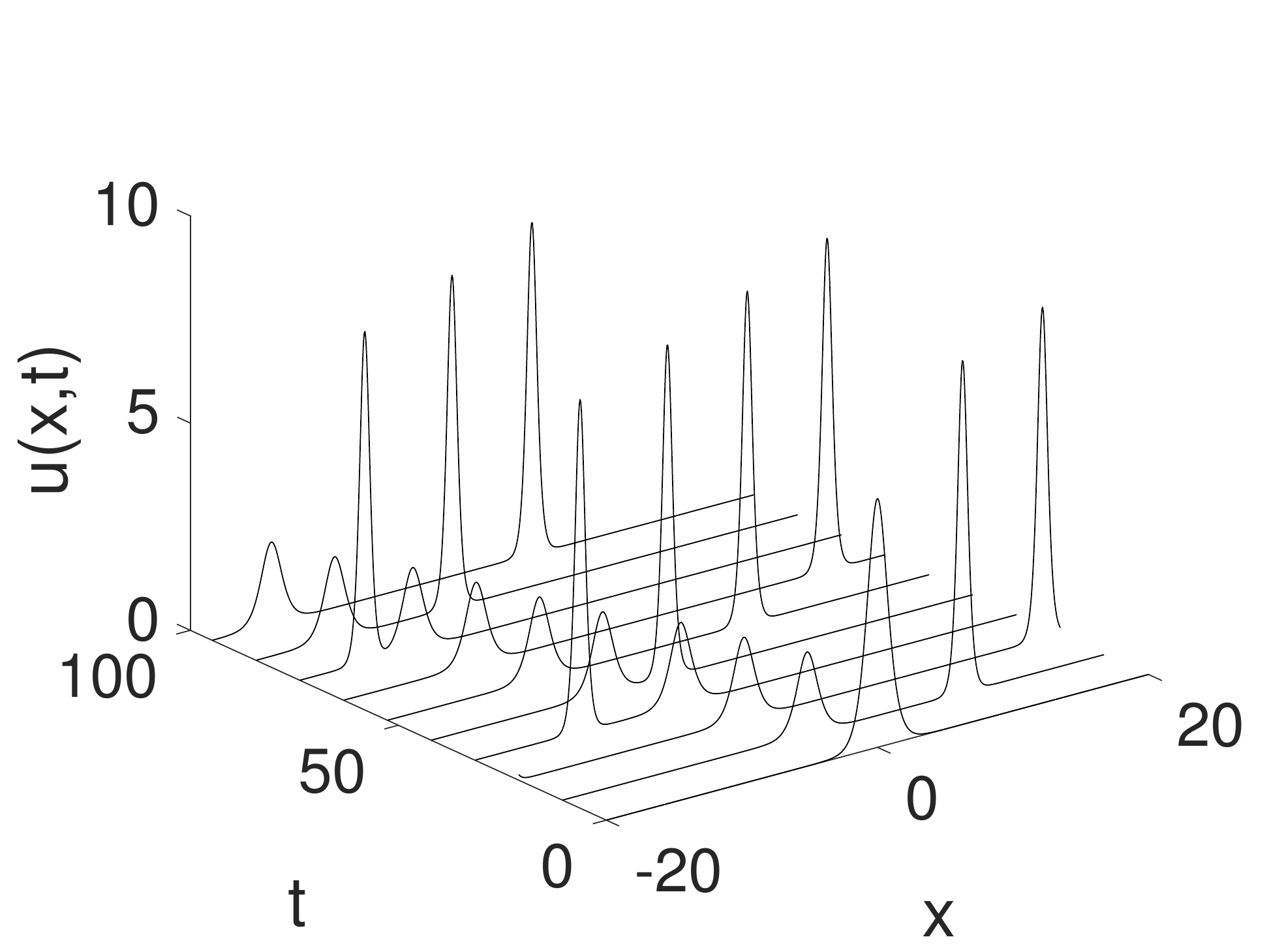}
        \end{subfigure}
 \begin{subfigure}[b]{0.33\textwidth}
      \centering
                \includegraphics[width=0.99\textwidth]{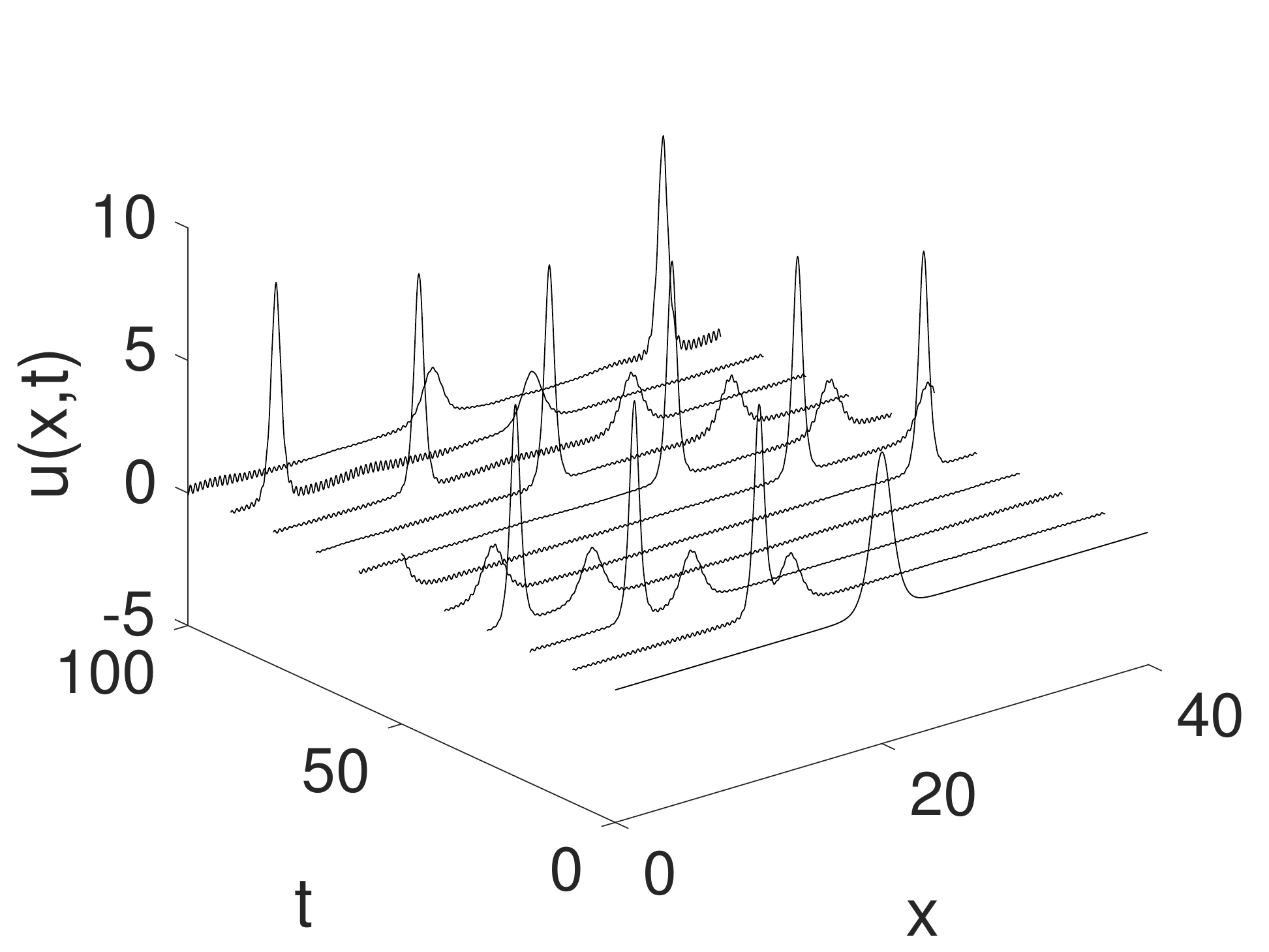}
        \end{subfigure}
      \caption{$\Delta x=0.05$. \textbf{Left:} Propagations of the wave by PDGM, $\Delta t=0.00375$, \textbf{middle:} propagations of the wave by Kahan's method, $\Delta t=0.00375$, \textbf{right:} Propagations of the wave by Kahan's method, $\Delta t=0.0125$.}\label{solution-kdv-2soliton-TKahan-LMI}
\vspace{10pt}
\end{figure}

\subsubsection {Dispersion analysis}
We consider the traditional linear analysis of numerical dispersion relations for the numerical schemes applied to the KdV equation, 
getting the dispersion relation of frequency $\omega$ and wave number $\xi$ to be
\begin{align}\label{eq_kdv_dispersion}
\omega&=\xi^3, \qquad \text{(exact solution)}\\
\textrm{sin}\omega&=\lambda (1-\textrm{cos}\xi)(3\textrm{cos}\omega-1)\textrm{sin}\xi, \qquad \text{(PDGM)}\\
\frac{\textrm{sin}\omega}{1+\textrm{cos}\omega}&=\lambda (1-\textrm{cos}\xi)\textrm{sin}\xi, \qquad \text{(Kahan)}
\end{align}
where $\lambda=\frac{\Delta t}{{\Delta x}^3}$.
The dispersion curve is displayed in Figure \eqref{solution-kdv-1soliton-TKahan-LMI} (right). We observe that Kahan's method is better than PDGM at preserving the exact dispersion relation. This coincides with the behaviour of the methods applied to the nonlinear KdV equation shown in Section \ref{Numericla example kdv}.

\section{Conclusion}
In this paper we perform a comparative study of Kahan's method and what we call the polarised discrete gradient (PDG) method. To that end, we present a general form encompassing a class of two-step methods that includes both a specific case of the PDG method and Kahan's method over two steps. We also compare the methods for two Hamiltonian PDEs: the KdV equation and the Camassa--Holm equation. Both Kahan's method and the PDG method are linearly implicit methods, which will save computational cost. 
A series of numerical experiments has been performed here, for the KdV equation with one and two solitons, and the Camassa--Holm equation with one and two peakons.
These experiments show that Kahan's method is more stable than the PDG method. They also indicate that Kahan's method yields more accurate results, as we have witnessed in the energy error and the shape and phase error when comparing to analytical solutions. Based on our results, we would recommend the use of Kahan's method if one seeks a linearly implicit scheme for a Hamiltonian system with $H$ cubic. 

\section*{Acknowledgements}
The authors would like to thank Elena Celledoni, Takayasu Matsuo and Brynjulf Owren for initiating the work that led to this paper, and for their very helpful comments on the manuscript. The first and second authors acknowledge support from the European Union Horizon 2020 research and innovation programmes, under the Marie Skłodowska-Curie grant agreement no. 691070.

\bibliography{LIM_Jan10}
\bibliographystyle{ieeetr}

\end{document}